\documentclass[10pt]{article} 
\usepackage{tikz}
\usepackage{tikz-cd}
\usepackage{xparse}
\usepackage{pgfplots}
\usetikzlibrary{matrix}  
\usetikzlibrary{decorations.markings}
\usepackage{boxedminipage}
\usepackage{scalerel,amsthm}
\usepackage{mathtools}
\usepackage{subfig}
\usepackage[lite]{amsrefs} 
\usepackage{amsmath,amssymb}
\usepackage[all]{xy} 
\usepackage{enumerate} 
\usepackage{enumitem}
\usepackage{hyperref}
\usepackage{graphicx} 
\usetikzlibrary{matrix}
\usepackage{xcolor}
\usetikzlibrary{backgrounds}

\numberwithin{equation}{section}
\usepackage[customcolors]{hf-tikz}
\usetikzlibrary{matrix,decorations.pathreplacing}
\pgfkeys{tikz/mymatrixenv/.style={decoration=brace,every left delimiter/.style={xshift=3pt},every right delimiter/.style={xshift=-3pt}}}
\pgfkeys{tikz/mymatrix/.style={matrix of math nodes,left delimiter=[,right delimiter={]},inner sep=2pt,column sep=1em,row sep=0.5em,nodes={inner sep=0pt}}}
\pgfkeys{tikz/mymatrixbrace/.style={decorate,thick}}

\newcommand\mydots{\hbox to 1em{.\hss.\hss.}}
\makeatletter
\newcommand*{\rom}[1]{\expandafter\@slowromancap\romannumeral #1@}
\makeatother

\newcommand{\subscript}[2]{$#1 _ #2$}

\newtheorem{theorem}{Theorem}[section]

\newtheorem{lemma}[theorem]{Lemma}
\newtheorem{proposition}[theorem]{Proposition}
\newtheorem{corollary}[theorem]{Corollary}
\theoremstyle{definition}
\newtheorem{example}[theorem]{Example}
\newtheorem{dfn}[theorem]{Definition}
\newtheorem{remark}[theorem]{Remark}

\ExplSyntaxOn
\NewDocumentCommand{\xhrulefill}{O{}}
 {
  \group_begin:
  \severin_xhrulefill:n { #1 }
  \group_end:
 }

\keys_define:nn { severin/xhrulefill }
 {
  height .dim_set:N    = \l_severin_xhrule_height_dim,
  thickness .dim_set:N = \l_severin_xhrule_thickness_dim,
  fill .skip_set:N     = \l_severin_xhrule_fill_skip,
  height .initial:n    = 0pt,
  thickness .initial:n = 0.4pt,
  fill .initial:n      = 0pt plus 1fill,
 }

\cs_new_protected:Nn \severin_xhrulefill:n
 {
  \keys_set:nn { severin/xhrulefill } { #1 }
  \leavevmode
  \leaders\hrule 
    height \dim_eval:n { \l_severin_xhrule_thickness_dim + \l_severin_xhrule_height_dim }
    depth  \dim_eval:n { -\l_severin_xhrule_height_dim }
  \skip_horizontal:N \l_severin_xhrule_fill_skip
  \kern 0pt
}
\ExplSyntaxOff

\DeclareMathOperator{\im}{im}

\newcommand{\unit}{1\!\!1}

\DeclareMathOperator{\ind}{ind}

\DeclareMathOperator{\vndet}{det_{\Gamma}}

\DeclareMathOperator{\vntr}{tr_{\Gamma}}

\DeclareMathOperator{\vnH}{\mathcal H}

\DeclareMathOperator{\vndim}{\dim_{\mathcal N(\Gamma)}}

\DeclareMathOperator{\GL}{GL}

\DeclareMathOperator{\reals}{\mathbb R}

\DeclareMathOperator{\ceals}{\mathbb C}
\DeclareMathOperator{\End}{End}
\DeclareMathOperator{\spec}{spec}

\DeclareMathOperator{\dom}{dom}

\DeclareMathOperator{\Int}{Int}

\DeclareMathOperator{\tr}{tr}
\DeclareMathOperator{\Isom}{Isom}
\DeclareMathOperator{\Cr}{Cr}

\DeclareMathOperator{\vnN}{\mathcal N}

\DeclareMathOperator{\sob}{\mathcal  W}

{}

\DeclareMathOperator{\bu}{*}

\usetikzlibrary{arrows.meta,bending}

\author{Benjamin Wa{\ss}ermann}
\title{An $L^2$-Cheeger M\"uller theorem on compact manifolds-with-boundary}
\begin{document}

\footnotetext{This work is funded by the DFG (Deutsche Forschungsgemeinschaft) -- 28186985}

\maketitle

\begin{abstract}
We generalize a Cheeger-M\"uller type theorem for flat, unitary bundles on infinite covering spaces over manifolds-with-boundary, proven by Burghelea, Friedlander and Kappeller. Employing recent anomaly results by Br\"uning, Ma and Zhang,
we prove an analogous statement for a general flat bundle that is only required to have a unimodular restriction to the boundary. 
\end{abstract}

\section{Introduction and statement of the main results}
For any flat bundle $E$ over a compact, triangulated manifold $M$ (briefly denoted by $E \downarrow M$ throughout this paper), one can construct the classical Reidemeister torsion, see for example \cite{Milnor:Torsion}. In \cite{Chapman}, Chapman showed that for acyclic bundles, this torsion is independent of the chosen triangulation, thereby also suggesting that there must be an alternative way to define it. \\  With the aid of a Riemannian metric $g$ on $M$, Ray and Singer defined in \cite{Ray:Atorsion} the analytic torsion for unitary bundles $E \downarrow M$ and showed that it does not depend on the choice of metric $g$ in case that $\partial M = \emptyset$. Furthermore, they conjectured that this analytic torsion must be equal to the Reidemeister torsion. This result was then independently proven by M\"uller \cite{Muller:Tor1} and Cheeger \cite{Cheeger:Tor} in the case $\partial M = \emptyset$. Later, M\"uller defined analytic torsion in the setting of a unimodular bundle $E \downarrow M$ and extended his earlier result \cite{Muller:Tor2}. At about the same time, Bismut and Zhang formulated a Cheeger-M\"uller type theorem for arbitrary flat bundles $E \downarrow M$ \cite{Bismut:Extension}, generalizing the notion of analytic and Reidemeister torsion in the same process. \\
The $L^2$-versions of Reidemeister and analytic torsion first appeared in \cite{Carey:Tor}, respectively \cite{Mathai:Tor}, and were first only defined for compact manifolds that are $L^2$-acyclic. Burghelea, Friedlander, Kappeler and MacDonald later extended these definitions onto unitary bundles $E \downarrow M$ without any assumption on $L^2$-acyclicity \cite{Friedlander:Uni}, and showed that both invariants are in fact equal. In \cite{Zhang:CM}, adapting the methods he earlier co-developed in \cite{Bismut:Extension}, Zhang extended this result even further onto arbitrary flat bundles, providing an explicit formula of the anomaly between $L^2$-Reidemeister and $L^2$-analytic torsion in this case and strengthening  an earlier result \cite{Friedlander:Rel} in the same vein by Burghelea, Friedlander and Kappeler. Instead of Reidemeister torsion, the authors of \cite{Bismut:Extension}, \cite{Friedlander:Uni}, \cite{Friedlander:Rel} and \cite{Zhang:CM} used the so-called Morse-Smale torsion (see Section $3.1$), which is defined using triangulations derived from a given Morse function $f \colon M \to \reals$. While additionally requiring a Hermitian bundle metric $h$ of $E$ to be defined, the smooth data involved in its construction makes it applicable for the Witten-deformation technique that plays a key part in the comparsion with analytic torsion. Moreover, although not explicitly written down anywhere, it is folklore knowledge that the Morse-Smale torsion coincides with the Reidemeister torsion whenever the volume form induced by $h$ is flat. An explicit proof of this and other related statements will be subject of a separate paper from the author. \\
Now assume that  $\partial M \neq \emptyset$. Under the assumption that the Hermitian metric $h$ is flat and the Riemannian metric $g$ is a product near $\partial M$, the difference between Reidemeister and Analytic torsion has been made explicit by L\"uck \cite{Lueck:Tor} and Vishik \cite{Vishik:Tor}. After various different generalizations of this particular result with relaxed assumptions on the metrics $g$ and $h$, the most general case without any further assumptions on $g$ or $h$ has been studied by Br\"uning and Ma in \cite{Bruning:Glue,Bruning:An}, who were able to prove an anomaly formula \cite[Theorem 0.1]{Bruning:Glue} entirely extending the result of Bismut and Zhang \cite[Theorem 0.2]{Bismut:Extension} to manifolds-with-boundary.\\
Adapting the techniques of their original result in the closed manifold case, the relation between $L^2$-Reidemeister torsion and $L^2$-analytic torsion on manifolds-with-boundary was studied by Burghelea, Friedlander and Kappeller \cite{Friedlander:Bd} under the assumption that $h$ is flat and $g$ is a product near $\partial M$. In \cite{Lueck:hyp}, L\"uck and Schick showed that anomaly of the $L^2$-analytic torsion is created when $g$ is deformed near $\partial M$. This anomaly was made explicit by Ma and Zhang \cite{Zhang:AN}, showing that it in fact equals the anomaly of ordinary analytic torsion. Making use of all the results mentioned so far, our main result, Theorem \ref{MAINTHEOREM2}, will be a Cheeger-M\"uller type theorem for $L^2$-acyclic unimodular bundles $E \downarrow M$ on manifolds-with-boundary satisfying $\chi(M) = 0$. \\ In order to state the result, we fix a flat bundle $E \downarrow M$ as above, along with a Riemannian metric $g$ on $M$ and an Hermitian metric $h$ on $E$. Additionally, we choose a Morse function $f \colon M \to \reals$, whose critical points lie in the interior of $M$ and that is constant along $\partial M$, together with some Riemannian metric $g'$ so that the pair $(f,g')$ satisfies the {\itshape Smale-transversality} conditions, c.f.\ Definition \ref{dfnms}. We denote by $\nabla_{g'}f$ the corresponding gradient vector field and call the quadruple $(E \downarrow M,g,h,\nabla_{g'}f)$ a {\itshape type \rom{2} Morse-Smale system}. We say that $(E \downarrow M,g,h,\nabla_{g'}f)$ is {\itshape of product form} if both $g$ and $h$ are products near $\partial M$, see Definition \ref{ADMISSIBLESYSTEM}. Provided that the bundle $E \downarrow M$ is of {\itshape determinant class} (Definitions \ref{MSCPLX},\ref{DRCPLX} \& Theorem \ref{ANCOMBCH}), a Ray-Singer analytic $L^2$-torsion \begin{equation} T^{RS}_{(2)}(E \downarrow M,g,h,\nabla_{g'}f) \in \reals_{>0}, \end{equation} as well as a Morse-Smale $L^2$-torsion \begin{equation} T^{MS}_{(2)}(E \downarrow M,h,\nabla_{g'}f) \in \reals_{>0}. \end{equation} can be defined, see Definitions \ref{MSCPLX} and Equation \ref{RSL2TOR}. To make precise the anomaly between the two $L^2$-torsions, two quantities need to be introduced. The first of these is given by the $1$-form \begin{equation} \theta(h) \in \Omega^1(M), \end{equation} derived from $h$, see Equation \ref{VOLCH}. Roughly stated, it measures the change along $M$ of the volume form induced by $h$, and vanishes precisely when $h$ is unimodular, i.e.\ when the volume form induced by $h$ is flat. The second one is the so-called {\itshape Mathai-Quillen current} \begin{equation} \Psi(TM,g) \in \Omega^{n-1}(TM \setminus M,\mathcal O_{TM}) \end{equation} derived from $g$ \cite[Definition 3.6]{Bismut:Extension}, where $\mathcal O_{TM}$ denotes the orientation line bundle over the tangent bundle $TM$. Since the gradient field $\nabla_{g'}f$ can also be regarded as a smooth embedding from $M \setminus \Cr(f)$ into $TM \setminus M$ (where $M \subseteq TM$ is identified with the zero section), we obtain via pullback a density \begin{equation} \nabla_{g'}f^* \Psi(TM,g) \in \Omega^{n-1}(M \setminus \Cr(f),\mathcal O_{M}). \end{equation} With all these objects introduced, our first main result can be stated as follows: \begin{theorem}[Theorem \ref{MAINTHEOREM}]  Let $\mathcal D = (E \downarrow M,g,h,\nabla_{g'}f)$ be a type \rom{2} Morse-Smale system of product form, where $M$ is an odd-dimensional manifold and $h|_{\partial M}$ is unimodular. Further, assume that both $E \downarrow M$ and $E|_{\partial M} \downarrow \partial M$ are of determinant class. Then 
\begin{align}& \log \left(\frac{T^{RS}_{(2)}(E \downarrow M,g,h,\nabla_{g'}f) }{T^{MS}_{(2)}(E \downarrow M,h,\nabla_{g'}f)}\right) = - \frac{\log 2}{4} \chi(\partial M)  \dim(E)   -  \frac{1}{2}\int_{M} \theta(h) \wedge \nabla_{g'}f^* \Psi(TM,g). \end{align}
\end{theorem}  This result can be viewed as a strict generalization of the main result of \cite{Friedlander:Bd}, where the authors made the more restrictive assumption that the metric $h$ is globally flat. \\
In order to state the second main result of this paper, we suppose that the bundle $E \downarrow M$ is unimodular, i.e.\ associated with a unimodular representation of $\pi_1(M)$. Then, assuming that $E \downarrow M$ is $L^2$-acyclic and of determinant class, one can define a {\itshape topological $L^2$-torsion} $T^{Top}_{(2)}(M,E) \in \reals_{>0}$. It can be defined similarly like $T^{MS}_{(2)}(E \downarrow M,h,\nabla_{g'}f)$, with the aid of {\itshape any} given CW-structure on $M$ and any fixed inner product on $V$, see \cite[Definition 5.2.5,~Theorem 5.3.12]{Ich}, and coincides with $T^{MS}_{(2)}(E \downarrow M,h,\nabla_{g'}f)$ whenever $h$ is unimodular. 
\begin{theorem}[Theorem \ref{MAINTHEOREM2}]  Let $(M,g)$ be a compact, connected, odd-dimensional Riemannian manifold. Then, there exists a density $B(g) \in \Omega^{n-1}(\partial M,\mathcal O_{\partial M})$ with $B(g) \equiv 0$ when $g$ is product-like near $\partial M$, such that the following holds:
\\ Let $E \downarrow M$ be a flat, finite-dimensional complex vector bundle, such that
\begin{enumerate}[label=\emph{$(\alph*)$}] 
\item $E$ is unimodular,
\item the pair $(M,E)$ is $L^2$-acyclic and of determinant class, 
\item the restriction $(\partial M,E|_{\partial M})$ is of determinant class.
\end{enumerate}
Then, for any choice of unimodular metric $h$ on $E$, one has
\begin{equation} \log \left (\frac{T^{RS}_{(2)}(E \downarrow M,g,h,\nabla_{g'}f) }{T^{Top}_{(2)}(M,E)} \right) = \frac{1}{2}\dim_{\ceals}(E) \int_{\partial M} B(g). \end{equation}  \end{theorem}
In a forthcoming paper, Theorem $1.2$ will be used to generalize the main result of \cite{Lueck:hyp} by L\"uck and Schick, in which we will show the equality of Ray-Singer analytic $L^2$-torsion and topological $L^2$-torsion for a large class of flat, unimodular bundles over finite-volume, hyperbolic manifolds, which are studied by several other authors as well \cite{Abert:Growth,Bergeron:Growth,Muller:Hyp,Muller:cover}. \\
The rest of this paper is subdivided into $6$ sections, which are structured as follows: In Section $2$, we briefly review the abstract theory of Hilbert $\vnN(\Gamma)$-modules that is necessary to define the Novikov-Shubin invariants, the determinant class condition and the general $L^2$-torsion which are studied in the rest of the paper. In Section $3$, we introduce the central objects of this paper: Morse-Smale systems, their Morse-Smale $L^2$-torsion and analytic $L^2$-torsion, as well as the derived metric $L^2$-torsion and relative torsion. In Section $4$, we state our main results, Theorem \ref{MAINTHEOREM} and \ref{MAINTHEOREM2} and give a proof of the latter. In Section $5$, we present product and anomaly formulas for the different $L^2$-torsions. In Section $6$, we will review the techniques employed by Burghelea et al.\ in their original proof for unitary bundles: Witten-Deformation, the splitting of the de Rham complex into the small and large subcomplex and the asymptotic expansions of the respective $L^2$-torsions. In Section $7$, we give a proof of Theorem \ref{MAINTHEOREM}. \\
The paper is based on parts of the author's dissertation \cite{Ich}. I thank my advisor, Prof.\ R.\ Sauer, for his constant support and numerous fruitful discussions. 
\section{$L^2$-torsion of Hilbert $\vnN(\Gamma)$-cochain complexes}\label{Hilsec}
We start by recollecting the objects and theory of Hilbert $\vnN(\Gamma)$-modules that are relevant for this paper. The well-acquainted reader may skip this section. \\
Throughout, we fix a countable group $\Gamma$. We define $L^2(\Gamma)$ to be the complex Hilbert space generated over the set $\Gamma$. Note that multiplying group elements of $\Gamma$ from the left naturally determines a left, linear, isometric $\Gamma$-action on $L^2(\Gamma)$. More generally, a complex Hilbert space $\vnH$ is called a {\bfseries Hilbert $\vnN(\Gamma)$-module} if it comes equipped with a left, linear, isometric $\Gamma$-action, so that there exists a $\Gamma$-linear, isometric embedding of $\vnH$ into $L^2(\Gamma) \hat \otimes H$ for some Hilbert space $H$. Here, $L^2(\Gamma) \hat \otimes H$ denotes the {\itshape Hilbert space tensor product} of $L^2(\Gamma)$ and $H$, with isometric $\Gamma$-action given by the action on the left factor. If one can choose $H$ to be finite-dimensional over $\ceals$, we call the Hilbert $\vnN(\Gamma)$-module $\vnH$ {\itshape finitely generated}.  \\ A $\Gamma$-linear, closed and densely defined operator $f: \vnH \to \vnH'$ between two Hilbert $\vnN(\Gamma)$-modules $\vnH$ and $\vnH'$ is called a {\itshape morphism of Hilbert $\vnN(\Gamma)$-modules} Since $\Gamma$ is fixed and implicit throughout this section, we will simply refer to such $f$ as a morphism. Any {\itshape positive}, bounded endomorphism $f: \vnH \to \vnH$ admits a natural {\itshape von Neumann trace} \begin{equation} \vntr(f) \in [0,\infty], \end{equation} \cite[Definition 1.8]{Lueck:book}  which satisfies $\vntr(f) < \infty$ whenever $\vnH$ is finitely generated. With this, we define the {\itshape von Neumann dimension} of $\vnH$ is \begin{equation} \vndim(\vnH) \coloneqq \vntr(\unit_{\vnH}). \end{equation} 
The adjoint $f^*$, the self-adjoint composition $f^*f$, as well as the absolute value $|f| \coloneqq \sqrt{f^*f}$ of a morphism $f$ are again morphisms. Similarly, if $E^{|f|}$ is the spectral measure associated with the positive, self-adjoint operator $|f|$ and if $p$ is a positive, essentially bounded Borel function defined over the spectrum $\sigma(|f|)$ of $|f|$, then \begin{equation} p(f) \coloneqq \int_{-\infty}^\infty p \; dE^{|f|}(\lambda) \end{equation} is a positive, bounded endomorpism, which is why $\vntr(p(f)) \in [0,\infty]$ is always well-defined. In particular, the family $\{\chi_{[0,\lambda]}(f)\}_{\lambda \in \reals_{\geq 0}}$ of spectral projections associated to $f$ further gives rise to a non-decreasing, right-continuous function \begin{equation} F_f(\lambda) \coloneqq \vntr(\chi_{[0,\lambda]}(f)) \in [0,\infty] \end{equation} in $\lambda \geq 0$, called the {\itshape spectral density function} of $f$. We say that $f$ is {\itshape Fredholm} if $F_f(\lambda) < \infty$ for all $\lambda \geq 0$. As a quantitative measurement of the spectral behaviour near $0$, the  {\itshape Novikov-Shubin invariant} $\alpha(f) \in [0,\infty] \cup \{\infty^+\}$ of a Fredholm morphism $f$ is defined as \begin{equation} \alpha(f) \coloneqq \begin{cases} \liminf_{\lambda \to 0^+} \dfrac{\ln \left(F(\lambda) - F(0)\right)}{\ln(\lambda)} & \text{if} \; F(\lambda) > F(0) \;\; \forall \lambda > 0, \\ \infty^+ & \text{else}.\end{cases} \end{equation} $\alpha(f)$ equals the (purely formal) symbol $\infty^+$ precisely when $|f|$ has a spectral gap at $0$. \\ Moreover, if $f$ is Fredholm, its spectral density determines a Borel measure $dF_f$ on $\reals_{\geq 0}$ in the canonical fashion. A Fredholm morphism $f$ is said to be {\itshape of determinant class} if \begin{equation} \int_{0+}^1 \log(\lambda) dF_f(\lambda) > - \infty. \end{equation} A morphism $f$ with $\alpha(f) > 0$ is always of determinant class, although the converse need not hold. If $f$ is a {\itshape bounded} morphism of determinant class, we can define its {\itshape Fuglede-Kadison determinant} $\vndet(f) \in \reals_{>0}$ as \begin{equation} \log(\vndet(f)) \coloneqq \int_{0+}^{||f||} \log(\lambda) dF_f(\lambda).\end{equation} 
A cochain complex
\begin{equation} (C^*,c^*) \colon 0 \to C^0 \xrightarrow{c^0} C^1 \xrightarrow{c^1} C^2 \xrightarrow{c^2} C_3 \xrightarrow{c^3} \dots, \end{equation} 
with each $C^i$ a Hilbert $\vnN(\Gamma)$-module and each $c^i$ a (not necessarily bounded) morphism of Hilbert $\vnN(\Gamma)$-modules is called a {\bfseries Hilbert $\vnN(\Gamma)$-cochain complex}. If all but finitely many of the $C^i$'s are trivial, each $C^i$ is finitely generated and each $c^i$ is bounded, then $(C^*,c^*)$ is {\itshape of finite type}.
A family $f^* \colon C^* \to D^* \colon (f^k \colon C^k \to D^k)_{k \in \mathbb N}$ of {\itshape bounded} morphisms is called a {\itshape morphism between the Hilbert $\vnN(\Gamma)$-cochain complexes} $(C^*,c^*)$ and $(D^*,d^*)$ if it additionally satisfies $f^*(\dom(c^*)) \subseteq \dom(d^*)$ and $f^{*+1} \circ c^* = d^* \circ f^*$ on $\dom(c^*)$. $f^*$ is called an {\itshape isomorphism} if each $f^k$ is an isomorphism.  \\
We say that two morphisms $f^*,g^*: C^* \to D^*$ between Hilbert $\vnN(\Gamma)$-cochain complexes are {\itshape chain homotopic} (written $f \simeq g$) if there exists a collection of bounded morphisms $K^*: C^* \to D^{*-1}$, satisfying\begin{align*}& K^*(\dom(c^*)) \subseteq \dom(d^{*-1}), \\  &f^* - g^* = K^{*+1}  c^* + d^{*-1}  K^* \hspace{.5cm} \text{on} \; \dom(c^*). \end{align*}  $K^*$ is called an {\itshape chain homotopy} between $f^*$ and $g^*$. 
Two Hilbert $\vnN(\Gamma)$-cochain complexes $(C^*,c^*)$ and $(D^*,d^*)$ are called {\bfseries chain homotopy equivalent} (written $C^* \sim D^*$) if there exists morphisms $f^*: C^* \to D^*$ and $g^*: D^* \to C^*$ such that $f^*  g^* \simeq \unit_{D^*}$ and $g^*  f^* \simeq \unit_{C^*}$. $f^*$ is called a {\bfseries chain homotopy equivalence} between $C^*$ and $D^*$ with {\itshape chain homotopy inverse $g^*$}. \\The (full) $L^2$-coholomology of a Hilbert $\vnN(\Gamma)$-cochain complex is the graded Hilbert $\vnN(\Gamma)$-module defined as \begin{equation} H^*(C^*) \coloneqq \bigoplus_{k=0}^\infty H^k(C^*), \; \; \; H^k(C^*) \coloneqq \ker(c^k) / \text{clos}(\im(c^{k-1})). \end{equation}
A cochain complex $(C^*,c^*)$ is {\itshape Fredholm} if all of the restricted morphisms $c^k|_{\im(c^{k-1})^\perp}$, $k \in \mathbb N_0$, are Fredholm. Observe that a complex $(C^*,c^*)$ of finite type is automatically Fredholm. For a Fredholm complex, we define its $k$-th Novikov-Shubin invariant $\alpha_k(C^*) \in [0,\infty] \cup \{\infty+\}$ as \begin{equation} \alpha_k(C^*) \coloneqq \alpha(c^k|_{\im(c^{k-1})^\perp}). \end{equation} A Fredholm complex is said to be {\itshape of determinant class} if all of the restricted morphisms $c^k|_{\im(c^{k-1})^\perp}$ are of determinant class. If $C^*$ is a determinant class {\itshape and} of finite type, we define its $L^2$-Torsion $T^{(2)}(C^*) \in \reals_{>0}$ as \begin{equation}
\log(T^{(2)}(C^*)) \coloneqq \sum_{k=0}^\infty (-1)^k \log(\vndet(c^k)). \end{equation}
\begin{proposition}\cite[Proposition 4.1]{Gromov:Novi}\label{PROP1} Let $(C^*,c^*)$ and $(D^*,d^*)$ be two cochain complexes of Hilbert $\vnN(\Gamma)$-modules and $f^* \colon C^* \to D^*$ a chain homotopy equivalence between them. Then, $f^*$ descends to an isomorphism of $L^2$-cohomologies \begin{equation} H^*(f^*) \colon H^*(C^*) \to H^*(D^*). \end{equation} Additionally, if both $C^*$ and $D^*$ are Fredholm, we have
\begin{enumerate}
\item $\alpha_k(C^*) = \alpha_k(D^*)$ for each $k \in \mathbb N_0$. 
\item $C^*$ is of determinant class if and only if $D^*$ is of determinant class. \end{enumerate}
\end{proposition}
\begin{proposition}\cite[Lemma 3.44]{Lueck:book}\label{PROP2} Let $(C^*,c^*)$ and $(D^*,d^*)$ be two cochain complexes of Hilbert $\vnN(\Gamma)$-cochain complexes, both of finite type and of determinant class. Further, let $f^* \colon C^* \to D^*$ be a chain isomorphism between them. Then \begin{equation}
\log(T^{(2)}(C^*)) - \log(T^{(2)}(D^*)) = \sum_{k=0}^{\infty} (-1)^{k}   \log ( \vndet(f^k) 
) -  \sum_{k=0}^\infty (-1)^k \log  \left( \vndet(H^k(f^k)) \right). \end{equation}
\end{proposition}

\section{Relative torsion}
We commence by introducing in order the main objects of this paper: Morse-Smale systems and their Morse-Smale, analytic, metric and relative $L^2$-torsion. \\
By a {\bfseries system} $\mathcal D = (E \downarrow M, g,h, X)$, we will always mean a set of data consisting of a flat, complex vector bundle $E \downarrow M$ over a smooth manifold $M$, along with a Riemannian metric $g$ on $M$, a Hermitian form $h$ on $E$ and $X$ either a vector field or a complex-valued function over $M$. \\
Given a uniform lattice $\Gamma < \Isom(M,g)$, such a system $\mathcal D$ is called {\bfseries $\Gamma$-invariant} if in addition, the isometric action of $\Gamma$ on $(M,g)$ leaves $X$ invariant and extends to an action of bundle isometries on the metric bundle $(E,h) \downarrow (M,g)$. Observe that $\Gamma$-invariant
systems on $M$ are precisely the lifts of systems defined over the compact quotient $M / \Gamma$.  \\
Throughout this chapter, we will frequently form {\itshape products} of systems: Given for $i=1,2$ two systems $(E_i \downarrow M_i,g_i,h_i, X_i)$ with $X_i$ either both vector fields or functions, one obtains a new system $(E_1 \hat \otimes E_2 \downarrow M_1 \times M_2, g_1 \oplus g_2, h_1 \hat \otimes h_2, X_1 + X_2)$, where $M_1 \times M_2$ is the product manifold equipped with the (direct) sum metric $g_1 \oplus g_2$,  $X_1 + X_2$ is the sum of the two vector fields or functions, and 
\begin{itemize} 
\item $E_1 \hat \otimes E_2 \downarrow M_1 \times M_2$ is defined to be the flat tensor product bundle $\pi_1^*E_1 \otimes \pi_2^*E_2 \downarrow M_1 \times M_2$, where $\pi_i: M_1 \times M_2 \to M_i$ denotes the projection onto the $i$-th factor. Here, the flat structure we choose is the canonical one induced by its flat factors $\pi^*_iE_i$. Moreover,  
\item $h_1 \hat \otimes h_2 \coloneqq \pi_1^* h_1 \otimes \pi_2^* h_2$ is the tensor product of the respective pullback Hermitian forms. 
\end{itemize} 
The main focus of our attention will be {\bfseries Morse-Smale systems}, which are by definition systems $\mathcal D =(E \downarrow M,g,h,\nabla_{g'}f)$ with $(f,g')$ a Morse-Smale pair, the latter of which we are now going to define: First of all, a pair $(f,g')$ with $f \colon M \to \reals$ a Morse function and $g'$ a Riemannian metric on $M$ is called a {\itshape Morse pair}. Let $\nabla_{g'}f \in \Gamma(TM)$ be the gradient vector field constructed from $f$ and $g'$ and let $\psi_t$ be the flow associated to the differential equation \begin{equation} \frac{\partial y}{\partial t} = - \nabla_{g'}f(y). \end{equation} Provided that both $f$ and $g'$ are lifted from a compact quotient, which we will assume throughout, it follows that $\psi_t$ is globally defined, i.e.\ for all $t \in \reals$. With $\Cr(f) \subset M$ denoting the set of critical points of $f$, define for each $p \in \Cr(f)$ the {\itshape stable}, respectively {\itshape unstable manifolds} \begin{align}
& W^-(p) \coloneqq  \{ x \in M \colon \lim_{t \to -\infty} \psi_t(x) = p  \}, \\ &
W^+(p) \coloneqq  \{ x \in M \colon \lim_{t \to +\infty} \psi_t(x) = p  \}. \end{align} Both $W^+(p)$ and $W^-(p)$ are smooth submanifolds of $M$, the latter being diffeomorphic to $\reals^{\ind(p)}$. Here, as everywhere else, $0 \leq \ind(p) \leq n$ denotes the {\itshape index} of the critical point $p$.  \begin{dfn}\label{dfnms} A Morse pair $(f,g')$ on $M$ is called a {\itshape Morse-Smale pair}, if all of the following conditions are satisfied: \begin{enumerate} \item For each pair $p,q \in \Cr(f)$, the manifolds $W^-(p)$ and $W^+(q)$ intersect transversally. \item $(f,g')$ is {\itshape locally trivial} at $\Cr(f)$. This means that: \begin{enumerate} \item For any $0 \leq k \leq n$ and any $p \in \Cr(f)$, there exists (pairwise disjoint) coordinate neighborhoods \[\phi_p: U_p  \to \begin{cases} \reals^n & \text{if} \; p \notin \partial M \\ \reals^n_{x_n \geq 0} & \text{if} \; p \in \partial M \end{cases} \] of $p$ with $\phi_p(p) = 0$  and such that we have \[(f \circ \phi_p^{-1})(x_1,\dots,x_n) = f(p) - \frac{1}{2}(x_1^2 + \dots x_{\ind(p)}^2) + \frac{1}{2}(x_{\ind(p)+1}^2 + \dots + x_n^{2}).\]
\item The pullback $\phi_p^*(g_{\reals^n})$ of the standard Euclidean metric on $\reals^n$ equals $g'|_{U_p}$. 
\end{enumerate} \end{enumerate}
If $\partial M \neq \emptyset$, we additionally assume that there exists $\kappa > 0$, along with a collar neighborhood $U$ of $\partial M$ and a diffeomorphism $\psi_{g'} \colon \partial M \times [0,\kappa) \to U$ coming from the normal exponential map induced by $g'$, so that either of the following two (mutually exclusive) conditions hold:
\begin{enumerate}[label=\roman*.] \item One has $(f \circ \psi_{g'})(p,t) = f|_{\partial M}(p) + t^2$ (in particular, $f|_{\partial M}$ is a Morse function on $\partial M$ with $\Cr(f|_{\partial M}) = \Cr(f) \cap \partial M$). In this case, we say that $(f,g')$ is {\bfseries of type \rom{1}}.
 \item One has $(f \circ \psi_{g'})(p,t) = b - t$ with $b = \max(f) \in \mathbb Z$ (In particular, $\Cr(f) \cap \partial M = \emptyset$). In this case, we say that $(f,g')$ is {\bfseries of type \rom{2}}.
 \end{enumerate}
\end{dfn}
It is a classic result that any compact manifold admits Morse-Smale pairs $(f,g')$, both of type \rom{1} and of type \rom{2}, see e.g.\ \cite[Theorem 6.6]{Hurt:Morse}. In fact, we will almost exclusively focus on type \rom{2} Morse-Smale pairs. That is because the methods employed to prove Theorem \ref{MAINTHEOREM} require that the critical points of a given Morse function are all interior, and thus only work for type \rom{2} Morse-Smale pairs. Conversely, the techniques used in other papers, whose results play an essential role in the proof of Theorem \ref{MAINTHEOREM2}, only work for type \rom{1} Morse-Smale pairs, which is why we have included them in the above definition.
\begin{dfn}\label{ADMISSIBLESYSTEM} A Morse-Smale system of the form $\mathcal D = (E \downarrow M, g,h, \nabla_{g'} f)$ will be called a {\bfseries type \rom{2} Morse-Smale system} if $(f,g')$ is a type \rom{2} Morse-Smale pair.
A type \rom{2} Morse-Smale system is {\bfseries of product form}, if
\begin{enumerate}[label=(\subscript{P}{{\arabic*}})]
\item {\itshape $g$ is a product near $\partial M$}: There exists a collar neighborhood $V$ of $\partial M$  that is the diffeomorphic image of the normal exponential map $\psi_g: \partial M \times [0,\epsilon) \to V$ induced by $g$, such that $\psi_g^*(g|_V) = g|_{\partial M} \oplus dt^2$, where $dt^2$ denotes the standard Euclidean metric on the half-open interval $[0,\epsilon)$.
\item The isometry $\psi_g$ further extends to a flat bundle isometry \[ \Psi: (E|_{\partial M} \hat \otimes \ceals \downarrow \partial M \times [0,\epsilon), h|_{\partial M} \hat \otimes 1_{\ceals}) \to (E|_{V} \downarrow V,h_V). \] Here, $E_{\ceals} \downarrow [0,\epsilon)$ is the trivial $1$-dimensional vector bundle over $[0,1)$ (with trivial connection), $E|_{\partial M} \hat \otimes E_{\ceals} \downarrow \partial M \times [0,\epsilon)$ denotes the flat, complex product bundle as introduced in the previous paragraph and $1_{\ceals}$ denotes the canonical constant Hermitian form on $E_{\ceals} \downarrow [0,\epsilon)$.
\end{enumerate}
A type \rom{2} Morse-Smale system of product form is called {\itshape weakly admissible}, if
\begin{enumerate}[label=(\subscript{A}{{\arabic*}})]
\item $M$ is compact. 
\item One has $g \equiv g'$ near $\Cr(f)$ and outside from a neighborhood of $\partial M$. 
\item The metric $h$ is parallel (see Definition \ref{unitunim}) in a neighborhood of $\Cr(f)$.
\end{enumerate}
Finally, a weakly admissible system $\mathcal D$ is called {\bfseries admissible} if the following extra compatibility condition is satisfied: 
\begin{enumerate}[label=(\subscript{A}{{4}})]
\item the restriction $h|_{\partial M}$ of $h$ to $\partial M$ is unimodular (see Definition \ref{unitunim}).
\end{enumerate}
A $\Gamma$-invariant system $\mathcal D =(E \downarrow M,g,h,\nabla_{g'}f)$ that is the lift of an admissible, respectively weakly admissible system on the compact quotient $M/\Gamma$ is called {\bfseries $\Gamma$-admissible}, respectively {\bfseries weakly $\Gamma$-admissible}. 
\end{dfn}
Observe that a weakly admissible system is a Morse-Smale system on a compact manifold $M$ with special {\itshape local} conditions on the Riemannian metric $g$ and Hermitian form $h$ near $\partial M$ and the critical points of $f$, while for an admissible system, we additionally demand a {\itshape global} condition on $h|_{\partial M}$. In particular, any flat bundle $E \downarrow M$ over a compact manifold fits into some weakly admissible system $\mathcal D = (E \downarrow M, g,h,\nabla_{g'}f)$ (by choosing an appropriate partition of unity), which can be chosen admissible if and only if the restriction bundle $E|_{\partial M} \downarrow \partial M$ is unimodular. \subsection{The Morse-Smale $L^2$-torsion $T^{MS}_{(2)}(E \downarrow M,h,\nabla_{g'}f)$}
Let $\mathcal D = (M,E,g,h,\nabla_{g'}f)$ be a Morse-Smale system with $M$ connected, $\widetilde{M}$ the universal cover of $M$ and $\widetilde{\mathcal D} = (\widetilde{M},\widetilde{E},\widetilde{g},\widetilde{h},\nabla_{\widetilde{g'}}f)$ the corresponding lifted system over $\widetilde{M}$. With $\Gamma \coloneqq \pi_1(M)$, it follows that $\widetilde{\mathcal D}$ is a $\Gamma$-invariant system. Let $\rho: \Gamma \to \GL(V)$ be the complex, finite-dimensional representation associated to the flat bundle $E \downarrow M$. Then, as a $\Gamma$-equivariant flat bundle, $\widetilde{E}$ is isomorphic to the trivial flat bundle $\widetilde{M} \times V$ with diagonal $\Gamma$-action given by $\gamma.(x,v) = (\gamma.x,\rho(\gamma)v)$. We fix one such isomorphism throughout. \\
As before, denote for each $p \in \Cr(\widetilde{f})$ by $W^-(p)$ and $W^+(p)$ the unstable, respectively stable manifold at $p$. Observe that we have $\gamma. W^-(p) = W^-(\gamma.p) \cong \reals^{\ind(p)}$ for each $\gamma \in \Gamma$, which allows us to fix a global orientation $O_p$ on each unstable manifold $W^-(p)$ in a $\Gamma$-invariant way. Together with the fact that $W^-(p)$ and $W^+(q)$ intersect transversely for each pair $p,q \in \Cr(\widetilde{f})$, we can construct as in \cite[Theorem 3.6]{Qin:Morse} integers $n(p,q) \in \mathbb Z$ for each pair $p,q \in \Cr(\widetilde{f})$ with $\ind(q) = \ind(p) + 1$, satisfying \begin{align} & n(p,q) = n(\gamma.p,\gamma.q) \; \; \forall \gamma \in \Gamma, \tag{\text{MS$1$}}\label{MS1} \\ & \forall p \in \Cr(\widetilde{f}) \colon \# \{q \in \Cr(\widetilde{f}) : \ind(q) = \ind(p) + 1 \wedge n(p,q) \neq 0 \} < \infty, \tag{\text{MS$2$}}\label{MS2}  \\& \forall q \in \Cr(\widetilde{f}) \colon \#\{p \in \Cr(\widetilde{f}): \ind(p) = \ind(q) - 1  \wedge n(p,q) \neq 0 \} < \infty,  \tag{\text{MS$3$}}\label{MS3} \\& \forall p \in \Cr(\widetilde{f}) \; \text{and} \; \forall q \in \Cr(\widetilde{f}) \; \text{with} \; \ind(q) = \ind(p) + 2 \colon \sum_{ \ind(r) = \ind(p)+1} n(p,r)n(r,q) = 0. \tag{\text{MS$4$}}\label{MORSEDIF}\end{align} In fact, under the conditions imposed on the pair $(f,g')$, it follows from \cite[Theorems 3.8,~3.9]{Qin:Morse} (see also \cite[Theorem 5.4.10 ,~Corollary 5.4.12]{Ich}) that \begin{enumerate} \item the set $\{W^-(p) \colon p \in \Cr(\widetilde{f}) \}$ is the collection of open cells of a $\Gamma$-CW-complex $X \subseteq \widetilde{M}$, so that \item the inclusion $X \hookrightarrow \widetilde{M}$ is a simple $\Gamma$-homotopy equivalence. Moreover, \item the integer $n(p,q)$ is precisely the degree of the attaching map of the cell $W^-(q)$ to the cell $W^-(p)$. \end{enumerate} Define $[O_p]$ to be the complex line generated by $O_p$ and the cochain complex of vector spaces \begin{align} C^*(\widetilde{M},\nabla_{\widetilde{g'}}\widetilde{f},\widetilde{E}) \coloneqq \bigoplus_{p \in \Cr(\widetilde{f})} [O_p] \otimes_{\ceals} V \; , \; C^{k}(\widetilde{M},\nabla_{\widetilde{g'}}\widetilde{f},\widetilde{E}) \coloneqq \bigoplus_{\ind(p) = k} [O_p] \otimes_{\ceals} V \end{align} with boundary map 
\begin{align} \partial^*_{MS}: C^*(\widetilde{M},\nabla_{\widetilde{g'}}\widetilde{f},\widetilde{E})\to C^{*+1}(\widetilde{M},\nabla_{\widetilde{g'}}\widetilde{f},\widetilde{E}) \nonumber  \end{align} being the unique $\ceals$-linear extension of the assignment \begin{align} \partial^*_{MS}( [O_p] \otimes v ) \coloneqq \sum_{\ind(q) = \ind(p) + 1} n(p,q) \cdot [O_q] \otimes v. \end{align} By \ref{MS2}--\ref{MORSEDIF}, $\partial^*_{MS}$ is well-defined and satisfies $\partial^{k+1}_{MS} \circ \partial^k_{MS} = 0$ for each $0 \leq k \leq n-1$. Furthermore, the respective $\Gamma$-actions on $\widetilde{M}$ and $V$ intertwine to produce a $\Gamma$-action on $C^*(\widetilde{M},\nabla_{\widetilde{g}'}\widetilde{f},\widetilde{E})$ given by \begin{equation}\label{GAMMAAC} \gamma.( [O_p] \otimes v ) \coloneqq [O_{\gamma.p}] \otimes \rho(\gamma)v. \end{equation} Due to \ref{MS1}, it follows that $\partial^*_{MS}$ is $\Gamma$-equivariant. Now recall the $\Gamma$-equivariant Hermitian form $\widetilde{h}$, which is part of the system $\widetilde{\mathcal D}$. Equipping the total space $C^*(\widetilde{M},\nabla_{\widetilde{g}'}\widetilde{f},\widetilde{E})$ with the inner product structure given by the direct sum of inner products induced by $\widetilde{h}$ at each fiber, the $\Gamma$-action \ref{GAMMAAC} becomes an action by isometries. Taking the corresponding $L^2$-completion, one obtains a Hilbert $\vnN(\Gamma)$-cochain complex of finite type, which we will denote by $C_{(2)}^*(\widetilde{M},\nabla_{\widetilde{g}'}\widetilde{f},\widetilde{E},\widetilde{h})$. In fact, each module $C_{(2)}^k(\widetilde{M},\nabla_{\widetilde{g}'}\widetilde{f},\widetilde{E},\widetilde{h})$ is isomorphic to $L^2(\Gamma)^{m_k} \otimes_{\ceals} V$, where $m_k \in \mathbb N$ is the number of $\Gamma$-cosets of the set $\{ p \in \Cr(\widetilde{f}) \colon \ind(p) = k \}.$ \begin{dfn}\label{MSCPLX} $C_{(2)}^*(\widetilde{M},\nabla_{\widetilde{g}'}\widetilde{f},\widetilde{E},\widetilde{h})$ is called the {\itshape $L^2$-Morse-Smale cochain complex} induced by the system $\mathcal D$. For $0 \leq k \leq n$, we define the $k$-th $L^2$-Morse-Smale cohomology \begin{equation} H^k_{(2)}(M,\nabla_{g'}f,E,h) \coloneqq H^k\left(C_{(2)}^*(\widetilde{M},\nabla_{\widetilde{g}'}\widetilde{f},\widetilde{E},\widetilde{h})\right) \end{equation} and the {\bfseries c-$L^2$-Betti number} of the pair $(M,E)$ \begin{equation} {\mathfrak b}^{(2)}_{k}(M,E) \coloneqq \vndim \left(H^k_{(2)}(M,\nabla_{g'}f,E,h)\right) \in [0,\infty) \end{equation} as the von Neumann dimension of the $L^2$-Morse-Smale cohomology (throughout, the prefix ''c'' stands for {\itshape combinatorial}) and similarly the $k$-th {\bfseries c-Novikov-Shubin invariant} \begin{equation} \alpha^{Top}_k(M,E) \coloneqq \alpha_k \left(C_{(2)}^*(\widetilde{M},\nabla_{\widetilde{g}'}\widetilde{f},\widetilde{E},\widetilde{h})\right). \end{equation} We say that $(M,E)$ is {\itshape c-$L^2$-acyclic} if ${\mathfrak b}^{(2)}_k(M,\rho) = 0$ for all $0 \leq k \leq n$. We say that $(M,E)$ is {\itshape of c-determinant class} if the complex $C^*_{(2)}(\widetilde{M},\nabla_{\widetilde{g}'}\widetilde{f},\widetilde{E},\widetilde{h})$ is of determinant class. If $(M,E)$ is of c-determinant class, we define the {\bfseries $L^2$-Morse-Smale torsion of the system $\mathcal D$} as \begin{equation} T^{MS}_{(2)}(E \downarrow M,h,\nabla_{g'}f) \coloneqq T \left(C_{(2)}^*(\widetilde{M},\nabla_{\widetilde{g}'}\widetilde{f},\widetilde{E},\widetilde{h}) \right) = \prod_{k=0}^{n} \vndet(\partial^{k}_{MS})^{(-1)^{k+1}} \in \reals_{>0}. \end{equation} \end{dfn} 
As mentioned previously, the Morse-Smale cochain complexes $C^*(\widetilde{M},\nabla_{\widetilde{g}_1'}\widetilde{f}_1,\widetilde{E})$ and $C^*(\widetilde{M},\nabla_{\widetilde{g}_2'}\widetilde{f}_2,\widetilde{E})$ coming from two distinct Morse-Smale systems $\mathcal D_1 = (M,E,g_1,h_1,\nabla_{g_1'}f_1)$ and  $\mathcal D'= (M,E, g_2,  h_2 \nabla_{g_2'} f_2)$ defined over a fixed pair $(M,E)$ are the cellular cochain complexes of two $\Gamma$-homotopy equivalent subcomplexes of $\widetilde{M}$. By picking a cellular approximation of an explicit homotopy equivalence, one can easily show that the $L^2$-Morse-Smale complexes $C^*_{(2)}(\widetilde{M},\nabla_{\widetilde{g}_1'}\widetilde{f}_1,\widetilde{E},\widetilde{h}_1)$ and $C^*_{(2)}(\widetilde{M},\nabla_{\widetilde{g}_2'}\widetilde{f}_2,\widetilde{E},\widetilde{h}_2)$ are chain homotopy equivalent. By Proposition \ref{PROP1}, it follows that the c-$L^2$-betti numbers ${\mathfrak b}^{(2)}_k(M,E)$, the c-Novikov-Shubin invariants $\alpha^{Top}_k(M,E)$, as well as the c-determinant class condition do not depend on the explicit choices of metrics and Morse Smale function. \\ On the other hand, the $L^2$-Morse-Smale torsion $T^{MS}_{(2)}(E \downarrow M,h,\nabla_{g'}f)$ {\itshape does} in general depend on the choices of Hermitian forms and Morse-Smale pairs (although it is entirely independent of the Riemannian metric on $M$). However, under the assumption that $E \downarrow M$ is a unimodular bundle, $\chi(M) = 0$, and that $(M,E)$ is c-$L^2$-acyclic and of of c-determinant class defined as above, there exists a {\itshape topological $L^2$-torsion} \begin{equation} T^{Top}_{(2)}(M,E) \in \reals_{\geq 0}. \end{equation} It can be defined similarly like $T^{MS}_{(2)}(E \downarrow M,h,\nabla_{g'}f)$, with the aid of {\itshape any} given CW-structure on $M$ and any fixed inner product on $V$, see \cite[Definition 5.2.5,~Theorem 5.3.12]{Ich}. The following key result establishes a connection between the {\itshape a priori} different Morse-Smale torsions that come from distinct Morse-Smale systems and $T^{Top}_{(2)}(M,E)$. \begin{theorem}\cite[Theorem 5.4.15]{Ich}\label{TOPTOR} Assume that $E \downarrow M$ is a unimodular bundle over a compact manifold and that $\chi(M) = 0$. Let $\mathcal D = (M,E,g,h,\nabla_{g'}f)$ be an associated Morse-Smale system with $h$ unimodular and assume that $E \downarrow M$ is c-$L^2$-acyclic and of c-determinant class. Then, one has \[ T^{Top}_{(2)}(M,E) = T^{MS}_{(2)}(E \downarrow M,h,\nabla_{g'}f). \]  \end{theorem}
\subsection{The analytic $L^2$-torsion $T^{An}_{(2)}(E \downarrow M,g,h)$}
For a Morse-Smale system $\mathcal D$ as above, we now explain the construction of the $L^2$-de Rham complex $\Omega_{(2)}^*(\widetilde{M},\widetilde{E},\widetilde{g},\widetilde{h})$ as well as the computation of the $L^2$-analytic torsion $T^{An}_{(2)}(E \downarrow M,g,h)$. None of these considerations will take the Morse-Smale pair $(f,g')$ into account. \\
To begin with, let  \begin{equation} \Omega^*(\widetilde{M},\widetilde{E}) \cong \Omega^*(\widetilde{M}) \otimes_{\ceals} \Gamma(\widetilde{E}) \end{equation} be the twisted de Rham complex of $\widetilde{E}$-valued forms, with differential \begin{equation} d^* \colon  \Omega^*(\widetilde{M},\widetilde{E}) \to \Omega^{*+1}(\widetilde{M},\widetilde{E}) \end{equation} induced by the flat connection on $\widetilde{E}$. Notice that the (fixed) flat identification $\widetilde{E} \cong \widetilde{M} \times V$ allows us to naturally identify $\Omega^*(\widetilde{M},\widetilde{E})$ with $\Omega^*(\widetilde{M}) \otimes_{\ceals} C^\infty(\widetilde{M},V)$. The canonical $\Gamma$-action on $\Omega^*(\widetilde{M})$ given by pullbacks and the natural $\Gamma$-action on $\Gamma(\widetilde{E}) \cong C^\infty(\widetilde{M},V)$ induced by the linear representation $\rho: \Gamma \to \GL(V)$ intertwine to produce a $\Gamma$-action on $\Omega^*(\widetilde{M},\widetilde{E})$, with respect to which $d^*$ becomes $\Gamma$-equivariant. 
For $x \in \mathcal{M}$, denote by $\langle \; \;, \;\; \rangle_x$ the inner product at the fiber vector space $(\Lambda^* T^*\widetilde{M} \otimes \widetilde{E})_x$ naturally derived from the pair $g$ and $h$. Let $\mu_g \in \Omega^n(\widetilde{M})$ be the volume form induced by $\widetilde{g}$. Restricting to the $\Gamma$-invariant subspace $\Omega^*_c(\widetilde{M},\widetilde{E})$ of compactly supported forms, the integration over the pointwise inner product \begin{align} \langle \;\; , \;\; \rangle \colon \Omega^*_c(\widetilde{M},\widetilde{E}) \times \Omega^*_c(\widetilde{M},\widetilde{E}) \to \ceals, \\ \langle f,g \rangle \coloneqq \int_{\widetilde{M}} \langle f(x),g(x) \rangle_x d\mu_g(x) \label{IP}\end{align} determines itself an inner product on $\Omega^*_c(\widetilde{M},\widetilde{E})$, with respect to which the $\Gamma$-action on $\Omega^*_c(\widetilde{M},\widetilde{E})$ is by isometries. \\
Let $T^*\partial \widetilde{M}$ the cotangent bundle over the boundary ${\partial \widetilde{M}}$. As usual, the Riemannian metric $g$ induces an {\itshape orthogonal} decomposition of the restricted cotangent bundle $T^*\widetilde{M}|_{\partial \widetilde{M}} = T^*\partial \widetilde{M} \oplus N^* \partial \widetilde{M}$, where $N^*\partial \widetilde{M} \downarrow \partial \widetilde{M}$ denotes the $1$-dimensional conormal bundle over $\widetilde{M}$. For each $x \in \partial \widetilde{M}$, each $0 \leq k \leq n$ and each $\omega \in \Omega^k(\widetilde{M},\widetilde{E})$, the vector $\omega(x) \in  (\Lambda^k T^*{\widetilde{M}} \otimes \widetilde{E})_x$ consequently decomposes orthogonally into a {\itshape tangential} and a {\itshape normal} part: 
\begin{equation} \omega(x) = \vec{t}\omega(x) + \vec{n}\omega(x) \in ( \Lambda^k T^*{\partial \widetilde{M}} \otimes \widetilde{E})_x  \oplus  (\Lambda^{k-1} T^*\partial \widetilde{M}  \otimes N^* \partial \widetilde{M}\otimes \widetilde{E})_x. \end{equation} Let \begin{align} &\delta^* \colon \dom^*(\delta^*) \to \dom^{*-1}(\delta^{*-1}), \\& \dom^*(\delta^*) \coloneqq \{ \sigma \in \Omega_c^*(M,E) \colon \vec{n}\sigma = 0 \} \end{align} be the formal adjoint of $d^*$ with respect to the inner product \ref{IP} and with absolute boundary conditions. Define the {\itshape Hodge-Laplacian} with absolute boundary conditions \begin{align}& \Delta_* \coloneqq \delta^{*+1} d^* + d^{*-1} \delta^* \colon \dom(\Delta_*) \to \dom(\Delta_*), \\& \dom(\Delta_*) \coloneqq \{ \omega \in \Omega^*_c(\widetilde{M},\widetilde{E}) \colon \vec{n}\omega = \vec{n} d^* \omega = 0 \} \subseteq\Omega^{*}(\widetilde{M},\widetilde{E}). \end{align} 
Let $\Omega^*_{(2)}(\widetilde{M},\widetilde{E}) = \Omega^*_{(2)}(\widetilde{M},\widetilde{E},\widetilde{g},\widetilde{h})$ be the $L^2$-completion of $\Omega^*_c(\widetilde{M},\widetilde{E})$ with regards to the previously defined inner product. Together with the extension of the isometric $\Gamma$-action on $\Omega_c(\widetilde{M},\widetilde{E})$ , $\Omega^*_{(2)}(\widetilde{M},\widetilde{E})$ becomes a Hilbert $\vnN(\Gamma)$-module (although not a finitely generated one). Moreover, the restricted operators $d^*$ and $\Delta_*$ each admit unbounded closed, $\Gamma$-equivariant extensions (denoted by the same symbol), which can therefore be regarded as morphisms between the corresponding Hilbert $\vnN(\Gamma)$-modules. We obtain a cochain complex of Hilbert $\vnN(\Gamma)$-modules \begin{equation} 0 \rightarrow \Omega_{(2)}^0(\widetilde{M},\widetilde{E}) \xrightarrow{d^0} \Omega_{(2)}^1(\widetilde{M},\widetilde{E}) \xrightarrow{d^1} \dots \xrightarrow{d^{n-1}} \Omega_{(2)}^n(\widetilde{M},\widetilde{E}) \rightarrow 0, \end{equation} called the {\itshape $L^2$-de Rham complex} induced by the system $\mathcal D$.\\
For each $0 \leq k \leq n$, the (closed extension of the) formal adjoint $\delta^k$ (with absolute boundary conditions) is in fact the Hilbert space adjoint of the differential $d^k$ \cite[Proposition 3.4.6]{Ich} . Furthermore, the (closed extension of the) Laplace operator $\Delta_k$ is positive and self-adjoint \cite[Theorem 3.4.1]{Ich}. With $t$ ranging over $\reals_{> 0}$, let $e^{-t\Delta_k} \colon \Omega^k_{(2)}(\widetilde{M},\widetilde{E}) \to \Omega^k_{(2)}(\widetilde{M},\widetilde{E})$ be the $1$-parameter, monotonically decreasing family of positive {\itshape heat operators} associated to $\Delta_k$, defined via the spectral theorem applied to $\Delta_k$. Each $e^{-t\Delta_k}$ is a {\itshape bounded} morphism of Hilbert $\vnN(\Gamma)$-modules that is also of trace class, i.e.\ satisfies $\vntr(e^{-t\Delta_k}) < \infty$. More precisely, each $e^{-t\Delta_k}$ possesses a smooth integral kernel $e^{-t\Delta_k}(\;\; , \;\;) \in C^\infty(\widetilde{M} \times \widetilde{M},\End(V))$, such that for any arbitrary fundamental domain $\mathcal F \subseteq \widetilde{M}$ for the $\Gamma$-action on $\widetilde{M}$, one has the equality \begin{equation} \vntr(e^{-t\Delta_k}) = \int_{\mathcal F} \tr(e^{-t\Delta_k} (x,x)) d\mu_g(x), \end{equation} see \cite[Proposition 4.16]{Atiyah:El}.
By dominated convergence, we obtain for each $0 \leq k \leq n$ that \begin{equation}\vndim(\ker(\Delta_k)) =  \lim_{t \to \infty} \vntr(e^{-t\Delta_k}) \in \reals_{\geq 0}. \end{equation} In fact, the closed subspace  $\ker(\Delta_k) \subseteq \Omega_{(2)}^k$ of {\itshape $L^2$-integrable harmonic $k$-forms} is not only a finite-dimensional Hilbert $\vnN(\Gamma)$-module, but also consists entirely of smooth forms and is isomorphic to the $k$-th $L^2$-cohomology \begin{equation} \mathcal H_{(2)}^k(M,E,g,h) \coloneqq H^k(\Omega^*(\widetilde{M},\widetilde{E},\widetilde{g},\widetilde{h})) \end{equation} of $\Omega^*(\widetilde{M},\widetilde{E},\widetilde{g},\widetilde{h})$ \cite[Propositions 3.4.2,~4.1.33]{Ich}. We define the $k$-th {\itshape a-$L^2$-Betti number} as\begin{equation} {\mathbf b}^{(2)}_k(M,E) \coloneqq \vndim(\mathcal H_{(2)}^k(M,E,g,h)) = \vndim(\ker(\Delta_k)).\end{equation} Throughout, the prefix ''a'' stands for {\itshape analytic}. \\ Similarly, the restriction $\Delta_k^\perp \coloneqq \Delta_k|_{\ker(\Delta_k)^\perp}$ is a self-adjoint morphism of Hilbert $\vnN(\Gamma)$-modules, so that \begin{equation} \vntr(e^{-t\Delta_k^\perp}) = \vntr(e^{-t\Delta_k}) - {\mathbf b}^{(2)}_{k}(M,E) \in \reals_{\geq 0} \end{equation} for each $t > 0$.
For $0 \leq k \leq n$ and $s \in \ceals$, the (truncated) {\itshape zeta-function} $\zeta_k(s)$ is defined as the formal expression \begin{equation} \zeta_k(s) \coloneqq \Gamma(s)^{-1}\int_{0}^{1} t^{s-1}\vntr(e^{-t\Delta_k^\perp}) dt. \end{equation} Here, $\Gamma(s)^{-1}$ denotes the (entire) inverse {\itshape gamma function}, which should not be confused with the {\itshape group} $\Gamma$. Due to the rational asymptotic behavior of $\vntr(e^{-t\Delta_k^\perp})$ near $t=0$ \cite[Lemma 1.3]{Lück:hyp} (see also \cite[Theorem 4.3.2]{Ich}), there exists a constant $C > 0$, such that $\zeta_k(s)$ determines a holomorphic function on the domain $\{s \in \ceals \colon \Re(s) >> C \}$ that extends to a meromorphic function on all of $\ceals$ with $s = 0$ being a regular point. 
\begin{dfn}\label{DRCPLX} Let $\mathcal D = (M,E,g,h,\nabla_{g'}f)$ be a Morse-Smale system as above. For $0 \leq k \leq n$, the $k$-th {\bfseries a-Novikov-Shubin} invariant $\alpha^{An}_{k}(M,E) \in [0,\infty] \cup \{\infty^+\}$ is defined as \begin{equation} \alpha_k^{An}(M,E) \coloneqq \alpha_k \left( \Omega^*_{(2)}(\widetilde{M},\widetilde{E},\widetilde{g},\widetilde{h}) \right). \end{equation} The pair $(M,E)$ is said to be {\bfseries of a-determinant class} if the $L^2$-de Rham complex $\Omega^*_{(2)}(\widetilde{M},\widetilde{E},\widetilde{g},\widetilde{h})$ is of determinant class. If $(M,E)$ is of a-determinant class, we can define the {\bfseries analytic $L^2$-torsion} $T^{An}_{(2)}(E \downarrow M,g,h) \in \reals_{> 0}$ of the system as 
\begin{equation} \log(T^{An}_{(2)}(E \downarrow M,g,h)) \coloneqq \sum_{k=0}^{n} \frac{k}{2}(-1)^{k+1} \left( \frac{d}{ds} \zeta_k(s)|_{s=0} + \int_{1}^\infty t^{-1} \vntr(e^{-t\Delta_k^\perp}) dt \right). \end{equation}
\end{dfn}
The a-determinant class condition of $(M,E)$ says that for each $0 \leq k \leq n$, the restriction $d^k|_{\im(d^{k-1})^\perp}$ is of determinant class. By \cite[Lemma 3.30]{Lueck:book}, this is equivalent to the operator $\Delta_k^\perp$ being of a-determinant class for each $0 \leq k \leq n$, which in turn \cite[Lemma 3.139]{Lueck:book} implies that $\int_{1}^\infty t^{-1} \vntr(e^{-t\Delta_k^\perp})dt < \infty$ for each $0 \leq k \leq n$, showing that $T^{An}_{(2)}(E \downarrow M,g,h)$ is well-defined. 
Up to bounded, $\Gamma$-equivariant isomorphisms, the Hilbert $\vnN(\Gamma)$-cochain complex $\Omega^*_{(2)}(\widetilde{M},\widetilde{E},\widetilde{g},\widetilde{h})$ is independent of the particular choice of $g$ and $h$. Therefore, neither ${\mathbf b}_k^{(2)}(M,E)$ nor the a-determinant class conditions depend on $g$ or $h$. However, in the general case that we concern ourselves with (i.e.\ when $\partial M \neq \emptyset$), the quantity $T^{An}_{(2)}(E \downarrow M,g,h)$ does depend on both $g$ and $h$. The precise metric anomalies, to be presentend in the next section, are of fundamental importance for this paper. 
\subsection{The metric $L^2$-torsion $T^{Met}_{(2)}(\mathcal D)$ and the relative $L^2$-torsion $\mathcal R(\mathcal D)$}
We now describe for a general Morse-Smale system $\mathcal D= (E \downarrow M,g,h,\nabla_{g'} f)$ with $M$ compact the construction of the {\bfseries relative $L^2$-torsion $\mathcal R(\mathcal D) \in \reals$}, provided that $E \downarrow M$ is determinant class.
To begin with, we are going to define new norms on $\Omega^*_c(\widetilde{M},\widetilde{E},\widetilde{g},\widetilde{h})$, the de Rham complex of compactly supported forms. Throughout, we will denote by $|| \; \; \;||_0$ the $L^2$-norm defined in the previous section. Assume first that that $\partial \widetilde{M} = \emptyset$. In this case, we define for each $s \in \reals_{>0}$ the $s$-th Sobolev norm \begin{equation} ||\omega||_s \coloneqq ||(1+\Delta_k)^{s/4}\omega||_0, \; \; \omega \in \Omega^k_c. \end{equation} In case that $\partial \widetilde{M} \neq \emptyset$, the define for each integer $p \in \mathbb N$ the $p$-th Sobolev norm (with absolute boundary conditions) inductively as \begin{equation} ||\omega||_p^2 \coloneqq ||\omega||_{p-1}^2 + ||d^k\omega||_{p-1}^2 + ||\delta^{k-1} \omega||_{p-1}^2 + ||\vec{n} \omega||_{p - 1/2}^2 \; \; \omega \in \Omega^k_c. \end{equation} For fixed $0 \leq k \leq n$ and integer $p \in \mathbb N_0$, the $L^2$-completion $\mathcal W^{k}_{p}(\widetilde{M},\widetilde{E},\widetilde{g},\widetilde{h})$ is called the {\itshape $p$-th Sobolev space of $k$-forms}. Just like in the case $p = 0$, the $\Gamma$-action on $\Omega^k_c$ extends to an isometric $\Gamma$-action on $W^k_p$, turning it into a Hilbert $\vnN(\Gamma)$-module. Crucially, we obtain {\itshape bounded} extensions \begin{align} & d^k \colon \mathcal W^k_{p+1}(\widetilde{M},\widetilde{E},\widetilde{g},\widetilde{h}) \to \mathcal W^{k+1}_{p}(\widetilde{M},\widetilde{E},\widetilde{g},\widetilde{h}) \end{align} for each $0 \leq k \leq n-1$, which is why we can define for fixed $l \geq n$ a cochain complex of Hilbert $\vnN(\Gamma)$-modules \begin{equation}\sob_{l-*}^*(\widetilde{M},\widetilde{E},\widetilde{g},\widetilde{h})  \colon \sob_{l}^0(\widetilde{M},\widetilde{E},\widetilde{g},\widetilde{h}) \xrightarrow{d^0} \sob_{l-1}^1(\widetilde{M},\widetilde{E},\widetilde{g},\widetilde{h}) \xrightarrow{d^1} \dots \xrightarrow{d^{n-1}} \sob^n_{l-n}(\widetilde{M},\widetilde{E},\widetilde{g},\widetilde{h}). \end{equation} 
Now recall the $L^2$-Morse-Smale complex $C^*_{(2)}(\widetilde{M},\nabla_{\widetilde{g}'}\widetilde{f},\widetilde{E},\widetilde{h})$ and fix an integer  $l > 3n/2 + 1$. Then, it follows from the Sobolev inequality that one has $\sigma \in C^1 \cap L^2$ for each $\sigma \in \mathcal W^k_l$ and each $0 \leq k \leq n$. Together with our fixed isomorphism $\widetilde{E} \cong \widetilde{M} \times V$, we deduce that for each $p \in \Cr(\widetilde{f})$ with $\ind(p) = k$, the integral $\int_{W^{-}(p)} \sigma \in V$ over the $k$-dimesional unstable manfiold $W^-(p)$ is well-defined. In fact, it holds that $\sum_{\ind(p) = k} ||\int_{W^{-}(p)} \sigma||^2_{\widetilde{h}_p} < \infty$, e.g.\ \cite[Lemma 3.2]{Dodziuk}. Therefore, we can define a map between graded Hilbert $\vnN(\Gamma)$-modules
 \begin{align}\label{DRINT} & \Int^* \colon  \mathcal W_{l-*}^*(\widetilde{M},\widetilde{E},\widetilde{g},\widetilde{h}) \to C^*_{(2)}(\widetilde{M},\nabla_{\widetilde{g}'}\widetilde{f},\widetilde{E},\widetilde{h}), \\& \Int^k(\sigma) \coloneqq \sum_{\substack{p \in \Cr(f) \\ \ind(p) =k}}  [O_p] \otimes \left(\int_{W^-(p)} \sigma \right)\hspace{.5cm} \sigma \in \sob^{k}_{l-k}, \end{align}  
given by integration of Sobolev forms over the unstable manifolds. 
By a result of Laudenbach \cite[Appendix, Proposition 6]{Bismut:Extension}, $\Int^*$ is a cochain map. Let \begin{equation} \pi^*: \ker(\partial^*_{MS}) \to \ker(\partial^*_{MS})/\text{clos}(\im(\partial^{*-1}_{MS})) \eqqcolon H_{(2)}^*(\widetilde{M},\nabla_{\widetilde{g}'}\widetilde{f},\widetilde{E},\widetilde{h}) \end{equation} be the projection of the kernel of the $L^2$-Morse-Smale boundary operator onto the corresponding $L^2$-Morse-Smale homology. By a theorem of Dodziuk \cite{Dodziuk}, extended by Schick \cite{Schick:Bounded2} to manifolds with boundary and by Shubin \cite{Shubin:coeff} to non-unitary bundles, the map \begin{equation}\label{MSINTG} \Theta^*: \mathcal \ker(\Delta_*) \to H_{(2)}^*(\widetilde{M},\nabla_{\widetilde{g}'}\widetilde{f},\widetilde{E},\widetilde{h}), \end{equation} defined as the restriction of $\pi^* \circ \Int^*$ onto the closed subspace $\ker(\Delta_*) \subseteq \mathcal W_{l-*}^{*}(\widetilde{M},\widetilde{E},\widetilde{g},\widetilde{h})$ of $L^2$-harmonic forms is an isomorphism of finitely generated Hilbert $\vnN(\Gamma)$-modules. In particular, \begin{equation} {\mathfrak b}_k^{(2)}(M,E) = {\mathbf b}_k^{(2)}(M,E) \;\; 0 \leq k \leq n,\end{equation} i.e.\ the combinatorial and analytical $L^2$-Betti numbers of the pair $(M,E)$ agree. From now on, since c-$L^2$-acyclicity is equivalent to a-$L^2$-acyclicity, we simply say that the pair $(M,E)$ is {\itshape $L^2$-acyclic} whenever either of the two equivalent conditions hold. 
The isomorphism $\Theta^*$ now also allows us to define the {\bfseries metric $L^2$-torsion} $T^{Met}_{(2)}(\mathcal D) \in \reals_{\geq 0}$ of the system $\mathcal D = (E \downarrow M,g,h,\nabla_{g'}f)$ as 
\begin{equation} \log T^{Met}_{(2)}(\mathcal D) \coloneqq \sum_{k=0}^\infty (-1)^{k} \log\vndet(\Theta^k) = \frac{1}{2} \sum_{k=0}^\infty (-1)^k \log \vndet((\Theta^k)^*\Theta^k). \end{equation}
Assuming that $E \downarrow M$ is of a-determinant class, we define the {\bfseries Ray-Singer $L^2$ Torsion} $ T^{RS}_{(2)}(\mathcal D) \in \reals_{\geq 0}$ as 
\begin{equation}\label{RSL2TOR} \log T^{RS}_{(2)}(\mathcal D) \coloneqq \log \left( \frac{T^{An}_{(2)}(E \downarrow M,g,h)}{T^{Met}_{(2)}(\mathcal D)}\right), \end{equation} 
Of course, if $\ker(\Delta_*) = \{0\}$, i.e.\ if $(M,E)$ is $L^2$-acyclic, then $T^{Met}_{(2)}(\mathcal D) = 1$, so that $T^{RS}_{(2)}(\mathcal D)  = T^{An}_{(2)}(E \downarrow M,g,h)$. If $(M,E)$ is both of combinatorial and of analytical determinant class, the {\bfseries relative $L^2$-torsion} $\mathcal R(\mathcal D) \in \reals$ of the corresponding Morse-Smale system $\mathcal D = (E \downarrow M,g,h,\nabla_{g'}f)$ can be defined as
\begin{equation} \mathcal R(\mathcal D) \coloneqq \log \left( \frac{T^{RS}_{(2)}(\mathcal D)}{T^{MS}_{(2)}(E \downarrow M,h,\nabla_{g'}f)} \right). \end{equation} 
We will show in Theorem \ref{ANCOMBCH} that the condition $E \downarrow M$ being of a-determinant class is equivalent to $E \downarrow M$ being of c-determinant class. Therefore, we are justified to say that $E \downarrow M$ is {\itshape of determinant class} whenever either determinant class condition (and therefore both) is satisfied.
\begin{remark} It should be mentioned that the relative torsion $\mathcal R(\mathcal D) \in \reals$ can be defined even if the corresponding bundle $E \downarrow M$ is not of determinant class. In that case, the individual terms $T^{RS}_{(2)}(E \downarrow M,g,h,\nabla_{g'}f)$ and $T^{MS}_{(2)}(E \downarrow M,g,h,\nabla_{g'}f)$ are not real numbers, but non-vanishing vectors in the same orientation class of a particular $1$-dimensional real vector space. Therefore, their quotient yields a {\itshape positive} real number, which is why $\mathcal R(\mathcal D)$, the logarithm of the quotient as above, is still well-defined. It can be shown that the main Theorem \ref{MAINTHEOREM} still holds in this case. We refer to \cite{Friedlander:Rel},  \cite{Zhang:CM} and \cite{Braverman:Det} for a detailed study of $L^2$-torsion without the determinant class conditions. \end{remark}
\section{Statement of the main results}
Using the terminology introduced in the previous section, we are going to formulate the main results, Theorem \ref{MAINTHEOREM} and \ref{MAINTHEOREM2}. \\
First, however, we also need to establish the notion of a {\itshape local quantity}:
Given two systems $\mathcal D_i =  (E_i \downarrow M_i, g_i,h_i, X_i)$, an isometry $\phi: (M_1,g_1) \to (M_2,g_2)$ between the underlying Riemannian manifolds that satisfies $\phi^*X_2 = X_1$ and extends to a flat bundle isometry $\Phi: (E_1,h_1) \to (E_2,h_2)$ is called an {\itshape isomorphism} between the systems.
\begin{dfn}[Local Quantity] An assignment of a form $\alpha = \alpha(\mathcal D) \in Y$, where either $Y = \Omega^n(M,\mathcal O_M)$, or $Y = \Omega^{n-1}(\partial M,\mathcal O_{\partial M})$ for any system $\mathcal D = (E \downarrow M, g,h, X)$ is called a {\bfseries local quantity} of $\mathcal D$ if it satisfies the following
compatibility conditions: \begin{enumerate}
\item For any open subset $U \subseteq M$, it holds that $\alpha(\mathcal D|_{U}) = \alpha(\mathcal D)|_{U}$. 
\item If $\phi: M_1 \to M_2$ is an isomorphism between two systems $\mathcal D_i = (E_i \downarrow M_i,g_i,h_i,X_i)$ (for $i=1,2$), then 
$\phi^* \alpha(\mathcal D_2) = \alpha(\mathcal D_1)$. 
\end{enumerate}
Here, as everywhere else, $\mathcal O_M \downarrow M$ is the (real) {\itshape orientation line bundle} over $M$. Elements of $\Omega^n(M,\mathcal O_{M})$ are called {\itshape densities}. 
\end{dfn}
For any system $\mathcal D =( E \downarrow M, g,h,\nabla_{g'}f)$ with $(f,g')$ a Morse-Smale pair, we will now construct a local quantity of the derived system  $\mathcal D = (E|_{M \setminus \Cr(f)} \downarrow M \setminus \Cr(f),g,h, \nabla_{g'} f)$ that constitutes an integral part in the analysis of the anomaly between $L^2$-Ray Singer and Morse-Smale torsion. \\
First off, as carefully explained and constructed by Bismut and Zhang in \cite[Section 3]{Bismut:Extension}, the Levi-Civita connection of the Riemannian metric $g$ gives rise to the {\bfseries Mathai-Quillen Current} \begin{equation} \Psi(M,g) \in \Omega^{n-1}(TM \setminus M, \mathcal O_{TM}). \end{equation} 
Here, we have identified $M \subseteq TM$ with its zero section inside $TM$. 
The second local quantity of relevance is is the $1$-form $\theta(h) \in \Omega^1(M)$, which measures the local change of the volume form induced the metric $h$ along $M$ and can be constructed as follows: Let $\nabla$ be the flat connection on $E \downarrow M$ and let $\overline{E^*} \downarrow M$ be the flat bundle over $M$ {\itshape conjugate dual} to $E \downarrow M$. The induced endomorphism bundle $\End(E,\overline{E^*}) \downarrow M$ carries a flat connection $\nabla^*$ naturally induced by $\nabla$. For the metric $h$, we now observe that $h \in \Gamma(M,\End(E,\overline{E^*}))$, which allows us to define the $1$-form \begin{equation}\label{VOLCH} \theta(h) \coloneqq \tr(h^{-1} \nabla^* h) \in \Omega^1(M). \end{equation} \begin{dfn}\label{unitunim} A metric $h$ on a flat bundle $E \downarrow M$ is called {\bfseries unitary} (or parallel) if $\nabla^*h \equiv 0$. $h$ is called {\bfseries unimodular} if $\theta(h) \equiv 0$. \end{dfn} The canonical metric associated to a flat unitary bundle $E \downarrow M$, i.e.\ every bundle coming from a unitary representation $\rho \colon \Gamma \to O(V)$), is unitary. Unitary metrics are obviously unimodular -- the converse need not hold. Every unimodular bundle $E \downarrow M$, i.e.\ every flat bundle corresponding to a finite-dimensional unimodular representation $\rho \colon \Gamma \to SL(V)$ admits a unimodular metric $h$. Although there is in general no canonical choice of a unimodular metric, such metrics can always be chosen with a lot of flexibility, as the next lemma shows: \begin{lemma}\cite[Corollary 5.4.18]{Ich}\label{UNIMODEX} Let $E \downarrow M$ be a flat, unimodular bundle over a connected manifold $M$ and $U = \bigsqcup_{i \in I} U_i \subseteq M$ a subset with each $U_i$ open and connected. Let $x_0 \in \Int(M \setminus U)$ and $x_i \in U_i$ for each $i \in I$ be chosen basepoints with curves $c_i \subseteq M$ connecting $x_0$ to $x_i$. 
Further, let $\widetilde{h}_0$ be a Hermitian metric on $E_{x_0}$ and $\widetilde{h_i}$ a Hermitian metric on $E_{x_i}$ satisfying \begin{equation} \det(\widetilde{h_i} \cdot P_{c_i}^*(\widetilde{h_0})^{-1}) = 1, \end{equation} where $P_{c_i}: \GL(E_{x_0},\overline{E_{x_0}^*}) \to \GL(E_{x_i},\overline{E_{x_i}^*})$ denotes the parallel transport along the curve $c_i$. Then, for any unimodular metric $\bigsqcup h_i$ on $E|_{U}$ extending $\bigsqcup \widetilde{h_i}$, there exists a global unimodular metric $h$ on $E$ further extending $\bigsqcup h_i \sqcup \widetilde{h}_0$. \end{lemma}  
Now notice that $\nabla_{g'} f$ determines a smooth {\itshape embedding} $\nabla_{g'}f\colon M \setminus \Cr(f) \to TM \setminus M$. Wedging the corresponding pullback $\nabla_{g'} f^* \Psi(M,g) \in \Omega^{n-1}(M \setminus \Cr(f),\mathcal O_M)$ with $\theta(h) \in \Omega^1(M)$, we obtain a density over $M \setminus \Cr(f)$ and local quantity of $\mathcal D$: \begin{equation} \theta(h) \wedge \nabla_{g'} f^* \Psi(M,g) \in \Omega^{n}(M \setminus \Cr(f), \mathcal O_M). \end{equation} 
This allows us to, at least formally, define the integral \begin{equation} \int_{M} \theta(h) \wedge \nabla_{g'} f^* \Psi(M,g) \coloneqq \int_{M \setminus \Cr(f)}  \theta(h) \wedge \nabla_{g'} f^* \Psi(M,g). \end{equation} 
Note that since $M \setminus \Cr(f)$ is not compact (unless $\Cr(f) = \emptyset$), the integral need {\itshape a priori} not converge. That this indeed always case has been shown in \cite{Bismut:Extension}, as an immediate consequence of their main result. Moreover, one can verify either from its explicit construction as done in \cite[Chapter \rom{3}]{Bismut:Extension} or immediately from \cite[Section 4]{Friedlander:Rel}, that $\theta(h) \wedge \nabla_{g'} f^* \Psi(M,g) $ is a local quantity of the system $\mathcal D = (E|_{M \setminus \Cr(f)} \downarrow M \setminus \Cr(f),g,h, \nabla_{g'} f)$, as claimed. The theorem that we wish to generalize is the following result by Zhang:

\begin{theorem}\cite[Theorem 4.2]{Zhang:CM}\label{CLOSINV2} Let $\mathcal D = (E \downarrow M,g,h,\nabla_{g'}f)$ be a system with $(f,g')$ a Morse-Smale pair and $M$ {\bfseries closed}. Then
\begin{equation}  \mathcal R(\mathcal D) = - \frac{1}{2} \int_{M} \theta(h) \wedge \nabla_{g'} f^* \Psi(M,g). \end{equation}
\end{theorem}

With aid of the above theorem, we will derive a similar result in case that $M$ is odd-dimensional with non-empty boundary: 

\begin{theorem}\label{MAINTHEOREM} Let $\mathcal D = (E \downarrow M,g,h,\nabla_{g'}f)$ be a type \rom{2} Morse-Smale system of product form, where $M$ is an odd-dimensional manifold and $h|_{\partial M}$ is unimodular. Further, assume that both $E \downarrow M$ and $E|_{\partial M} \downarrow \partial M$ are of determinant class. Then 
\begin{align}& \mathcal R(\mathcal D) = - \frac{\log 2}{4} \chi(\partial M)  \dim(E)   -  \frac{1}{2}\int_{M} \theta(h) \wedge \nabla_{g'}f^* \Psi(TM,g). \end{align}
\end{theorem}

\begin{remark} Similarly as in the unitary case (cf.\ \cite[Theorem 4.1]{Friedlander:Bd}), there is also a version of Theorem \ref{MAINTHEOREM} for relative/mixed, instead of absolute boundary conditions as we assume here throughout. The proof presented here carries over to this case with only minor modifications. Although not relevant for this paper, this generalization will prove to be useful when one wants to extend the gluing formula \cite[Theorem 4.3]{Friedlander:Bd} to non-unitary bundles, which could in turn be used for future computational purposes. \end{remark}

\begin{example}\label{TRIVEX}
Set $I = [a,b]$, and let $E_{\ceals} \coloneqq \ceals \times I$ be the trivial $1$-dimensional complex vector bundle over $I$. As metrics, we choose $g_0$ to be the standard Euclidean metric and $h_0$ the canonical constant Hermitian form, i.e $\langle z,z' \rangle_{h_0(x)} \coloneqq z \overline{z'}$ for any $x \in I$ and any pair $z,z' \in \ceals$. Further, we choose as  Morse-function a smooth map $f_0: [a,b] \to \reals$ satisfying \begin{itemize} \item $f_0(x) \coloneqq \frac{1}{2}(x - (b+a)/2)^2$ away from a neighborhood of $\{a,b\}$, \item $f_0(a + t\epsilon) = f_0(b - t \epsilon) = b - t\epsilon$ for all $t \in [0,1]$ and some small $\epsilon > 0$, and so that \item $(b+a)/2$ is the only critical point of $f_0$. \end{itemize} One now easily verifies that $\mathcal D_I \coloneqq (E_{\ceals} \downarrow I,g_0,h_0,\nabla_{g_0'}f_0)$ is an admissible system and that $E_{\ceals} \downarrow I$ is of determinant class. In fact, we can directly compute the corresponding analytic and combinatorial torsion elements. This computation will also be essential for the proof of Theorem \ref{MAINTHEOREM}.
Firstly, since $f_0$ has by construction only one critical point, the corresponding Morse-Smale complex has only one non-trivial chain module, immediately implying that
\begin{align}\label{Itop} \log T^{MS}_{(2)}(I,g_0,h_0,\mathcal{F},f_0) = 0. \end{align} 
Similarly, it follows that the de Rham integration map \begin{equation*} \Int^*: \Omega^*(I,\mathcal{F}) \to C_{MS}^*(I,g_0,h_0,\mathcal{F},f_0) = \ceals \otimes \left[\frac{b+a}{2} \right] \end{equation*} is only non-trivial on $\Omega^0(I,\mathcal{F}) \cong C^\infty(I,\ceals)$, on which it is defined by 
\begin{equation*} \Int^0(f) =  f \left(\frac{b+a}{2}\right) \otimes  \left[\frac{b+a}{2}\right]. \end{equation*} Therefore, the isomorphism \begin{equation*} \Theta^0: \mathcal H^0(I,\mathcal{F}) \to \ceals \otimes \left[\frac{b+a}{2}\right], \end{equation*} obtained by simply restricting $\Int^0$ to the space of harmonic, i.e.\ {\itshape constant}, functions, maps the function $f \equiv c$  to $c \otimes  [\frac{b+a}{2}]$. Since the inner product on $\ceals \otimes [\frac{b+a}{2}]$ in the canonical one determined by $h_0$ and the inner product $\mathcal H^0(I,\mathcal{F})$ is induced by integration over the interval $I = [a,b]$, it follows that the adjoint \begin{equation*} (\Theta^0)^*:\ceals \otimes \left[\frac{b+a}{2}\right]  \to \mathcal H^0(I,\mathcal{F}) \end{equation*} sends $c \cdot [\frac{b+a}{2}]$ to the constant function $f \equiv c(b-a)^{-1}$. Therefore, the composition $(\Theta^0)^*\Theta^0$ is simply scalar multiplication by $(b-a)^{-1}$, from which we deduce that \begin{align} \log T^{Met}_{(2)}(\mathcal D_I) = \frac{1}{2} \log\left(\det((\Theta^0)^*\Theta^0)\right) =  - \frac{1}{2} \log(b-a). \end{align}  In order to compute the analytic torsion, observe first that, 
under the isometric identification $\Omega^1(I,\mathcal F) \cong C^\infty(I,\ceals)$ with $f(x) dx \mapsto f(x)$, the Laplacian $\Delta_1$ defined over $\Omega^1(I,\mathcal F)$ corresponds to the closure of the elliptic operator $- \frac{\partial^2}{\partial x^2}$ with initial domain $\{ g \in C^\infty: g \equiv 0 \; \text{on} \; \{a,b\} \}$. 
It is well-known, see for example \cite[Section 4.2]{Strauss:PDG} for each $n \in \mathbb N_0$ that \begin{equation*} \spec(\Delta_1) = \spec(- \frac{\partial^2}{\partial x^2}) = \{ \frac{n^2\pi^2}{l^2} : n \in \mathbb N_0 \}, \end{equation*} with $l \coloneqq b - a$ (and eigenspace of $n^2\pi^2/l^2$ the $\mathbb C$-span of $\sin(n\pi/l (x-a))$. Therefore, the Zeta function $\zeta_{\Delta_1}(s)$ of $\Delta_1$ satisfies 
\begin{equation*} \zeta_{\Delta_1}(s) = \sum_{n=1}^\infty \left(\frac{l}{n\pi}\right)^{2s}  = \left(\frac{l}{\pi}\right)^{2s}  \sum_{n=1}^\infty \left(\frac{1}{n}\right)^{2s} = \left(\frac{l}{\pi}\right)^{2s} \cdot \zeta(2s), \end{equation*} where $\zeta$ denotes the ordinary Riemann Zeta-function. Applying the well-known equalities
$\zeta(0) = - \frac{1}{2}$ and $\zeta'(0) = - \frac{1}{2} \log(2\pi)$, we can thus compute
\begin{equation}\label{Ian} \log T^{An}_{(2)}(\mathcal D_I) = \frac{1}{2} \zeta_{\Delta_1}'(0)  = - \frac{1}{2} \left( \log(2) + \log(b-a) \right). \end{equation}
From \ref{Itop}--\ref{Ian}, we get
\begin{equation} \mathcal R(\mathcal D_I) = - \frac{\log 2}{2} = -\frac{\log{2}}{4}\chi(\{a,b\}) -\frac{1}{2} \int_{a}^b \overbrace{\theta(h_0)}^{=0} \wedge (\nabla_{g_0'} f_0)^* \Psi(TI,g_0). \end{equation}
\end{example}

The main part of this paper is devoted to the proof of \ref{MAINTHEOREM}. We will adapt the techniques and strategy developed by Burghelea, Friedlander and Kappeler in \cite{Friedlander:Bd} to our situation of non-unitary bundles, together with employing several known anomaly results that have been shown since. We remark that Theorem \ref{MAINTHEOREM} has also recently been verified in an (as of now) unpublished paper by Guangxiang Su, employing techniques and methods different from the ones that we are using. 
Theorem \ref{MAINTHEOREM}, together with the main results established by Br\"uning and Ma in \cite{Bruning:Glue}, Zhang and Ma in \cite{Zhang:AN}, and Zhang in \cite{Zhang:CM}, are then used to prove the next key result of this paper: 
\begin{theorem}\label{MAINTHEOREM2} Let $(M,g)$ be a compact, connected, odd-dimensional Riemannian manifold. Then, there exists a density $B(g) \in \Omega^{n-1}(\partial M,\mathcal O_{\partial M})$ with $B(g) \equiv 0$ when $g$ is product-like near $\partial M$, such that the following holds:
\\ Let $E \downarrow M$ be a flat, finite-dimensional complex vector bundle, such that
\begin{enumerate}[label=\emph{$(\alph*)$}] 
\item $E$ is unimodular,
\item the pair $(M,E)$ is $L^2$-acyclic and of determinant class, 
\item the restriction $(\partial M,E|_{\partial M})$ is of determinant class.
\end{enumerate}
Then, for any choice of unimodular metric $h$ on $E$, one has
\begin{equation} \log \left (\frac{T^{An}_{(2)}(E \downarrow M,g,h)}{T^{Top}_{(2)}(M,E)} \right) = \frac{1}{2}\dim_{\ceals}(E) \int_{\partial M} B(g). \end{equation}
In particular, for $i=1,2$ and any two representations $E_i \downarrow M$ satisfying the above assertions, it follows that
\begin{equation} \dim_{\ceals}(E_2)  \log \left( \frac{ T^{An}_{(2)}(E_1 \downarrow M,g,h_1)}{T^{Top}_{(2)}(M,E_1)} \right) = \dim_{\ceals}(E_1)  \log \left( \frac{ T^{An}_{(2)}(E_2 \downarrow M,g,h_2)}{T^{Top}_{(2)}(M,E_2)} \right), \end{equation}
for any choice of unimodular metric $h_i$ on $E_i \downarrow M$. 
\end{theorem}

\begin{remark} Observe that the statement is vacuous in the case that $M$ possesses no flat bundle $E \downarrow M$ so that $(M,E)$ is $L^2$-acyclic. In particular, this is true whenever $\chi(M) \neq 0$, cf.\ \cite[Theorem 1.35]{Lueck:book}. \end{remark}

\begin{proof} Let $\rho$ be a representation satisfying the assumptions from the theorem. By the previous remark, we must have \begin{equation}\label{zeroeul} 0 = \chi(M) = \frac{1}{2} \chi(\partial M), \end{equation} where the last equality follows since $M$ is odd-dimensional and compact.\\ Choose a Morse function $f$ on $M$ of type \rom{2}, along a Riemannian metric $g'$ on $M$ that is a product near $\partial M$ and so that $(f,g')$ is a Morse-Smale pair. By Lemma \ref{UNIMODEX}, we may also choose a unimodular metric $h'$ with $h'|_{\partial M} \equiv h|_{\partial M}$ and so that $\mathcal D = (E \downarrow M,g',h',f)$ becomes an admissible system (in particular, $h'$ is of product form near $\partial M$). First, since $h'$ is unimodular and $E \downarrow M$ is det-$L^2$-acyclic, we obtain from Theorem \ref{TOPTOR} that \begin{align}\label{MSTOPEQ}  T^{MS}_{(2)}(E \downarrow M,h',\nabla_{g'} f) = T^{Top}_{(2)}(M,E). \end{align}  
Furthermore, we can apply \ref{zeroeul} and Theorem \ref{MAINTHEOREM} to this situation and obtain
\begin{align} \log\left(\frac{T^{An}_{(2)}(E \downarrow M,g',h',f)}{T^{MS}_{(2)}(E \downarrow M,h',\nabla_{g'}f)}\right) = \mathcal R(\mathcal D)= 0.  \end{align} .  \\
Next, choose a type \rom{1} Morse function $f': M \to \reals$  on $M$. As $E \downarrow M$ is by assumption $L^2$-acyclic, we have $T^{An}_{(2)}(E \downarrow M,g,h) = T^{RS}_{(2)}(E \downarrow M,g,h,f')$ and analogously $T^{An}_{(2)}(E \downarrow M,g',h') = T^{RS}_{(2)}(E \downarrow M,g',h', f')$. 
Moreover, by the main result of \cite{Zhang:AN}, we have the equality of Ray-Singer anomalies \begin{align}\label{ANAN} &  \log \left(\frac{T^{An}_{(2)}(E \downarrow M,g,h)}{T^{An}_{(2)}(E \downarrow M,g',h')}\right) = \log \left(\frac{T^{RS}_{(2)}(E \downarrow M,g,h,f')}{T^{RS}_{(2)}(E \downarrow M,g',h',f')}\right) = \log \left(\frac{T^{RS}(E \downarrow M,g,h,f')}{T^{RS}(E \downarrow M,g',h', f')}\right). \end{align}
Here, $T^{RS}(E \downarrow M,g',h')$ is the (ordinary) Ray-Singer-metric as originally introduced in \cite[Definition 2.2]{Bismut:Extension} and first extended to manifolds with boundary in \cite{Bruning:An}.
Further, it is shown in \cite[Theorem 3.4]{Bruning:Glue} that there exists a density $B(g) \in \Omega^{n-1}(\partial M,\mathcal O_{\partial M})$ with $B(g) \equiv 0$ whenever $g$ is also product-like near $\partial M$, so that 
\begin{equation}\label{ANMST}  \log \left(\frac{T^{RS}(E \downarrow M,g,h,f')}{T^{RS}(E \downarrow M,g',h', f')}\right) = \frac{1}{2} \dim_{\ceals}(E) \int_{\partial M} B(g). \end{equation}
The density $B(g)$ is constructed as in \cite[Page 1103]{Bruning:Glue}. It depends only on the local geometry of $(\partial M,g|_{\partial M})$ inside $(M,g)$. \\
Using \ref{MSTOPEQ} -- \ref{ANMST} , we finally obtain
\begin{align}& \log\left(\frac{T^{An}_{(2)}(E \downarrow M,g,h)}{T^{Top}_{(2)}(M,E)}\right) = \log \left(\frac{T^{An}_{(2)}(E \downarrow M,g,h)}{T^{An}_{(2)}(E \downarrow M,g',h')}\right)  +  \log\left(\frac{T^{An}_{(2)}(E \downarrow M,g',h')}{T^{MS}_{(2)}(E \downarrow M,h',\nabla_{g'} f)}\right) \nonumber \\
& = \log \left(\frac{T^{RS}(E \downarrow M,g,h,f')}{T^{RS}(E \downarrow M,g',h',f')}\right)  =  \frac{1}{2}\dim_{\ceals}(E) \int_{\partial M}B(g), \end{align}
as desired. 
\end{proof}

\section{Product formulas, determinant class and subdivisions}
In this section, we study the effect on $L^2$-torsion and the local quantities after having taking the product of two systems. Moreover, we will make precise the anomaly of relative torsion that occurs when taking a subdivision of a Morse function and appropriate new metrics.  \\
As hinted towards in the introduction, given two Morse-Smale systems $\mathcal D_i = (E_i \downarrow M_i,g_i,h_i, \nabla_{g'_i} f_i)$ for $i=1,2$, an integral part of our methods will involve considering the product system $\mathcal D_1 \times \mathcal  D_2 = (E_1 \hat \otimes E_2 \downarrow M_1 \times M_2,g_1 \times g_2, h_1 \hat \otimes h_2, \nabla_{g_1' \times g_2'} (f_1 + f_2))$ and derive meaningful information of $\mathcal D_1 \times \mathcal D_2$ in terms of $\mathcal D_1$ and $\mathcal D_2$, and vice versa. Throughout, we assume exclusively that $M_1$ has non-empty boundary and $M_2$ has empty boundary. In this case, a problem that we have to address is that {\itshape a product of two type  \rom{2} Morse-Smale systems need not be a type \rom{2} Morse-Smale system anymore}.\\ The problem is due to the fact that the Morse function $f_1 + f_2$ doesn't necessarily fulfil condition ii of Definition \ref{dfnms} anymore (in particular, it is not necessarily constant on the boundary $\partial (M_1 \times M_2) = \partial M_1 \times M_2$). This can be remedied by deforming $f_1 + f_2$ in a sufficiently small neighborhood of $\partial M_1 \times M_2$ to be of the type \rom{2} shape as described in Definition \ref{ADMISSIBLESYSTEM}, which can be arranged in such a way that the resulting Morse function, denoted henceforth by $\underline{f_1 + f_2}$,  equals $f_1 + f_2$ outside of a small neighborhood of $\partial M_1 \times M_2$, has the same critical points as $f_1 + f_2$, the same gradient trajectories with respect to $\nabla_{g_1' + g_2'}$ and the same {\itshape unstable cells}. We denote the resulting {\itshape modified product system} by \begin{align}\label{MODIPROD} \underline{\mathcal D_1 \times \mathcal D_2} \coloneqq (E_1 \hat \otimes E_2 \downarrow M_1 \times M_2,g_1 \times g_2, h_1 \hat \otimes h_2, \nabla_{g_1' \times g_2'} (\underline{f_1 + f_2})), \end{align} and observe that $\underline{\mathcal D_1 \times \mathcal D_2}$ is of product form, respectively weakly admissible whenever both $\mathcal D_1$ and $\mathcal D_2$ are of product form, respectively weakly admissible.
Moreover, under the assumption that both $M_1$ and $M_2$ are compact, it follows immediately from the construction of $\underline{f_1 + f_2}$ that the Morse-Smale cochain complexes corresponding to $\underline{\mathcal D_1 \times \mathcal D_2}$  and $\mathcal D_1 \times \mathcal D_2$  are the same (as Hilbert $\vnN(\Gamma)$-cochain complexes). This immediately implies that \begin{align}& \label{METWARP} \log T^{Met}_{(2)}(\mathcal D_1 \times \mathcal D_2) = \log T^{Met}_{(2)}(\underline{\mathcal D_1 \times \mathcal D_2}). \end{align} In case that $E \downarrow M$ is of determinant class, we also get \begin{align}& \label{MSWARP} \log T^{MS}_{2}(\mathcal D_1 \times \mathcal D_2)  = \log T^{MS}_{(2)}(\underline{\mathcal D_1 \times \mathcal D_2}), \\&\label{ANWARP} \log T^{An}_{(2)}(\mathcal D_1 \times \mathcal D_2) = \log T^{An}_{(2)}(\underline{\mathcal D_1 \times \mathcal D_2}).\end{align}
Still, to obtain an admissible system from two admissible systems $\mathcal D_1$ and $\mathcal D_2$, we need to ensure that $h_1 \hat\otimes h_2$ is unimodular near $\partial M_1 \times M_2$, which can only be guaranteed if we assume additionally that $h_2$ is (globally) unimodular. For our purposes, this will provide no restriction at all, since we will always form products, where $E_2 \downarrow M_2$ is in fact a unitary bundle and $h_2$ is an associated unitary (and flat) metric.
Summarizing, we have the following: 
\begin{lemma} For, $i=1,2$, let $\mathcal D_i = (E_i \downarrow M_i,g_i,h_i,\nabla_{g_i'} f_i)$ be two type \rom{2} Morse-Smale systems with $\partial M_1 \neq \emptyset$ and $\partial M_2 = \emptyset$. Then, the modified product system system $\underline{\mathcal D_1 \times \mathcal D_2}$ as in \ref{MODIPROD} is also a type \rom{2} Morse-Smale system. Moreover, if both $\mathcal D_1$ and $\mathcal D_2$ are additionally of product form/weakly admissible, then also $\underline{\mathcal D_1 \times \mathcal D_2}$ is of product form/weakly admissible. Lastly, if both $\mathcal D_1$ and $\mathcal D_2$ are admissible, so that $h_2$ is globally unimodular, then $\underline{\mathcal D_1 \times \mathcal D_2}$ is also admissible.  \end{lemma}
The first product formula that we state is as follows is as follows
\begin{proposition}[Product Formula 1]\label{PRODFORM2}
For $i=1,2$, let $\mathcal D_i  = (E_i \downarrow M_i,g_i,h_i, \nabla_{g'_i} f_i)$ be two type \rom{2} Morse-Smale systems with $M_1$ compact, $\partial M_1 \neq \emptyset$ and with $M_2$ closed. Then, the type \rom{2} Morse-Smale system $\underline{\mathcal D_1 \times \mathcal D_2}$ is also of determinant class and we get
\begin{enumerate} \item $\log T^{An}_{(2)}(\underline{\mathcal D_1 \times \mathcal D_2}) = \chi(M_1,E_1) \log T^{An}_{(2)}(\mathcal D_2) + \log T^{An}_{(2)}(\mathcal D_1) \chi(M_2,E_2)$,
\item $ \log T^{Met}_{(2)} (\underline{\mathcal D_1 \times \mathcal D_2}) = \chi(M_1, E_1) \log T^{Met}_{(2)}(\mathcal D_2) +  \log T^{Met}_{(2)}(\mathcal D_1) \chi(M_2,E_2)$,
\item  $\log T^{MS}_{(2)} (\underline{\mathcal D_1 \times \mathcal D_2}) = \chi(M_1,E_1) \log T^{MS}_{(2)}(\mathcal D_2) + \log T^{MS}_{(2)}(\mathcal D_1) \chi(M_2,E_2)$,
\item $\mathcal R(\underline{\mathcal D_1 \times \mathcal D_2}) = \chi(M_1,E_1) \mathcal R(\mathcal D_2) + \mathcal R(\mathcal D_1) \chi(M_2,E_2)$.
\end{enumerate}
\end{proposition}
\begin{proof} $(1) - (3)$: If we replace $\underline{\mathcal D_1 \times \mathcal D_2}$ by the genuine product system $\mathcal D_1 \times \mathcal D_2$, the equalities are well-known. Namely, the proofs presented in \cite[Proposition 1.21, Proposition 4.2]{Friedlander:Bd} can be copied line by line, after changing the definition of $\Lambda^{-,q}(M,E)$ to be the $C^\infty$-closure of $d^*_q\left(\Omega^{q+1}(M,\partial M,E)\right)$. Now apply \ref{METWARP}-\ref{ANWARP}. $(4)$ is an immediate consequence of $(1) - (3)$. 
\end{proof} 
In addition, we will need to analyze the behavior under taking products of the local quantities introduced in the previous section. Here, the assumption that the Hermitian forms are unimodular at the boundary becomes essential. \\
For this, note first that we have a natural embedding $\Omega^*(M_1) \otimes \Omega^*(M_2) \hookrightarrow \Omega^*(M_1 \times M_2)$ (which is dense under the natural $C^\infty$-topology). By passing to local trivializations over coordinate charts, one easily sees that the $1$-form $\theta(h_1 \hat \otimes h_2)$ lies  in $\Omega^*(M_1) \otimes \Omega^*(M_2)$ and 
is of the form \begin{equation}\label{prodvolch} \theta(h_1 \hat \otimes h_2) = \theta(h_1) \otimes \dim(E_2) + \dim(E_1) \otimes \theta(h_2). \end{equation} 
Furthermore, it has been shown in \cite[pages 63-64]{Friedlander:Rel} (see also \cite[Chapter 4]{Bismut:Extension} or \cite[Theorem 2.7]{Bismut:Groth} for additional details) that 
\begin{align}& \nabla_{g_1' \times g_2'}(f_1 + f_2)^*\Psi (T(M_1 \times M_2), g_1 \times g_2) ) = (\nabla_{g_1'}f_1)^*\Psi(TM_1,g_1) \otimes  e(TM_2,g_2) \nonumber \\\label{prodmatq} &  +  e(TM_1,g_1) \otimes (\nabla_{g_2'}f_2)^* \Psi(TM_2,g_2) \end{align} on $M_1 \times M_2 \setminus \Cr(f_1 + f_2) = M_1 \times M_2 \setminus \Cr(f_1) \times C(f_2)$. 
Here, for a Riemannian manifold $(M,g)$, the {\bfseries Euler form} $e(M,g) \in \Omega^{\dim(M)}(M,\mathcal O_{M})$ is a density defined using Chern-Weil theory. It has the property that $e(M,g) \equiv 0$ whenever $M$ is odd-dimensional. Moreover, if $M$ is closed, it is a representative of the Euler class of the tangent bundle $TM \downarrow M$. By the Gauss-Chern-Bonnett theorem, it then follows that \begin{equation} \int_{M} e(M,g) = \chi(M), \end{equation} if $M$ is closed. We refer \cite[Page 1103]{Bruning:Glue} for an explicit formula for $e(M,g)$. \\
Combining \ref{prodvolch} with \ref{prodmatq}, we get
\begin{align} & \theta(h_1 \hat \otimes h_2) \wedge \nabla_{g_1' \times g_2'}(f_1 + f_2)^*\Psi (T(M_1 \times M_2), g_1 \times g_2)   = \theta(h_1) \wedge (\nabla_{g_1'}f_1)^*\Psi(TM_1,g_1)  \otimes \dim(E_2)e(TM_2,g_2)  \nonumber \\ \label{prod1} & + \dim(E_1)e(TM_1,g_1) \otimes  \theta(h_2) \wedge (\nabla_{g_2'}f_2)^*\Psi(TM_2,g_2)  \end{align}
on $M_1 \times M_2 \setminus \Cr(f_1 \times f_2)$. Here, we have used that $\theta(h_i) \wedge e(TM_i,g_i) \in \Omega^{\dim(M_i)+1}(M_i,\mathcal O_{M_i}) = \{0\}$ for both $i=1,2$. 
\begin{lemma}[Product Formula 2] \label{fac1} For $i =1,2$, let $\mathcal D_i \coloneqq (E_i \downarrow M_i,g_i,h_i, \nabla_{g'_i} f_i)$ be two type \rom{2} Morse-Smale systems of product form, so that both $h_1|_{\partial M}$ and $h_2$ are unimodular. Then, it holds that 
\begin{align} & \theta(h_1 \hat\otimes h_2) \wedge \nabla_{g_1' \times g_2'}(\underline{f_1 + f_2})^*\Phi(T(M_1 \times M_2),g_1 \times g_2), \nonumber \\ &=  \theta(h_1) \wedge (\nabla_{g_1'}f_1)^*\Psi(TM_1,g_1) \otimes \dim(E_2) \cdot e(TM_2,g_2) \end{align}
on all of $M \setminus \Cr(\underline{f_1 + f_2})$. In particular, if either $M_2$ is odd-dimensional or $h_1$ is also unimodular, then 
\begin{align}\label{vanish2} & \theta(h_1 \hat\otimes h_2) \wedge \nabla_{g_1' \times g_2'}(\underline{f_1 + f_2})^*\Phi(T(M_1 \times M_2),g_1 \times g_2)  = 0. \end{align} \end{lemma}
\begin{proof} Due to the assumption that $h_1|_{\partial M_1}$ and $h_2$ both are unimodular, it follows from \ref{prodvolch} that  $h_1|_{\partial M_1} \hat\otimes h_2$ determines a unimodular metric on the restriction bundle $E|_{\partial(M_1 \times M_2)} = E|_{\partial M_1 \times M_2}$. Since the system $\mathcal D_1$ is of product form, this allows us to choose a small neighborhood $U$ of $\partial M_1$, so that $\theta(h_1)\equiv 0$ on $U$. Together with Equation \ref{prodvolch} and $\theta(h_2) \equiv 0$ everywhere on $M_2$, we deduce that
\begin{equation}\label{vanish3} \theta(h_1 \hat\otimes h_2) \equiv 0 \hspace{1cm} \text{on} \; U \times M_2. \end{equation} 
By choosing $U$ smaller, if necessary, we also have by construction $\underline{f_1 + f_2} = f_1 + f_2$ on $(M_1 \setminus U )\times M_2$, and therefore the equality of gradients
\begin{equation}\label{equal2} \nabla_{g_1' \times g_2'}(\underline{f_1 + f_2}) = \nabla_{g_1' \times g_2'}(f_1 + f_2) \hspace{1cm} \text{on} \; (M_1 \setminus U) \times M_2. \end{equation}
The result now follows from \ref{vanish3}, \ref{equal2} and the product formula \ref{prod1}. 
\end{proof} 
Apart from considering products of systems, we will also have to investigate in the anomaly of the relative torsion that arises when changing the metrics of a given system. In fact, we will only look at  anomalies under the assumption that the metrics are left unchanged in a neighborhood of $\partial M$. The proposition below covers this situation, generalizing \cite[Proposition 5.1,5.2]{Friedlander:Rel} onto odd-dimensional Manifolds with boundary with product metrics near $\partial M$. 
\begin{proposition}[Metric anomaly with boundary conditions]\label{METANOREAL} Let $\mathcal D_i = (E \downarrow M, g_i,h_i, \nabla_{g} f)$ for $i=1,2$ be two Morse-Smale Systems with $M$ odd-dimensional, such that either \begin{enumerate}\item near $\partial M$, $g_1 \equiv g_2$ are of product form
and $h_1|_{\partial M} \equiv h_2|_{\partial M}$, or \item near $\partial M$, $g_1$ and $g_2$ are of product form and $h_1|_{\partial M} \equiv h_2|_{\partial M}$ is unimodular. \end{enumerate} Then 
\begin{equation}\mathcal R(\mathcal D_1)-\mathcal R(\mathcal D_2)= \sum_{p \in \Cr(f)} (-1)^{\ind(p)} \log \left( \det( h_1(p)^{-1} \circ h_2(p) ) \right). \end{equation}
\end{proposition}
\begin{proof} First, observe that
\begin{align}\label{ANO1}& \mathcal R(\mathcal D_1) - \mathcal R(\mathcal D_2) = \log \left(\frac{ T^{An}_{(2)}(\mathcal D_1)}{T^{An}_{(2)}(\mathcal D_2)}\right) + \log \left(\frac{T^{Met}_{(2)}(\mathcal D_2)}{T^{Met}_{(2)}(\mathcal D_1)}\right) +\log\left(\frac{T^{MS}_{(2)}(\mathcal D_2)}{T^{MS}_{(2)}(\mathcal D_1)}\right). \end{align}
Furthermore, we have
\begin{align}\label{METANO} \frac{T^{Met}_{(2)}(\mathcal D_2)}{T^{Met}_{(2)}(\mathcal D_1)} = \sum_{k=0}^n (-1)^k \log\left( \frac{\vndet(\Theta_2^k)}{\vndet(\Theta_1^k)}\right), \end{align} where $\Theta_i^*: \mathcal H^*(\widetilde{M},\widetilde{g_i},\widetilde{E},\widetilde{h_i}) \to H_{(2)}^*(\widetilde{M},E,h_i,\nabla_{g}f)$ are the isomorphisms of finitely generated Hilbert $\vnN(\Gamma)$-modules as defined in \ref{MSINTG}. We let \begin{equation} \unit_{[h_1,h_2]}^*:  H_{(2)}^*(\widetilde{M},\nabla_{\widetilde{g}}\widetilde{f},\widetilde{E},\widetilde{h_1}) \to H_{(2)}^*(\widetilde{M},\nabla_{\widetilde{g}}\widetilde{f},\widetilde{E},\widetilde{h_2}) \end{equation} be the isomorphism of Hilbert $\vnN(\Gamma)$-modules induced by the (not necessarily unitary) identity map $\unit_{[h_2,h_1]}^*: C^*_{(2)}(\widetilde{M},\nabla_{\widetilde{g}}\widetilde{f},\widetilde{E},\widetilde{h_2}) \to C^*_{(2)}(\widetilde{M},\nabla_{\widetilde{g}}\widetilde{f},\widetilde{E},\widetilde{h_1})$. Also, we let \begin{equation}\tau^*: \mathcal H^*(\widetilde{M},\widetilde{g_2},\widetilde{E},\widetilde{h_2}) \to \mathcal H^*(\widetilde{M},\widetilde{g_1},\widetilde{E},\widetilde{h_1})\end{equation} be the isomorphism of Hilbert $\vnN(\Gamma)$-modules making the diagram below commute. 
\begin{equation} \begin{tikzcd}  \mathcal H^*(\widetilde{M},\widetilde{g_1},\widetilde{E},\widetilde{h_1}) \ar{r}{\Theta_1^*} & H_{(2)}^*(\widetilde{M},\nabla_{\widetilde{g}}\widetilde{f},\widetilde{E},\widetilde{h_1}) \ar{d}{\unit_{[h_1,h_2]}^*}  \\ \mathcal H^*(\widetilde{M},\widetilde{g_1},\widetilde{E},\widetilde{h_2})\ar{u}{\tau^*} \ar{r}{\Theta_2^*} & H_{(2)}^*(\widetilde{M},\nabla_{\widetilde{g}}\widetilde{f},\widetilde{E},\widetilde{h_2})\end{tikzcd} \end{equation}
From the multiplicativity of the Fuglede-Kadison determinant \cite[Theorem 3.14]{Lueck:book}, it follows that
\begin{equation} \vndet(\tau^*) \vndet(\unit_{[h_1,h_2]}^*) = \vndet(\Theta_1^*)^{-1} \vndet(\Theta_2^*). \end{equation}
Therefore, Equation \ref{METANO} decomposes into
\begin{equation}\label{ANO2} \log \left(\frac{T^{Met}_{(2)}(\mathcal D_2)}{T^{Met}_{(2)}(\mathcal D_1)}\right) = \sum_{k = 0}^n (-1)^k \log \det(\tau^k)  + \sum_{k=0}^n (-1)^k \log \det(\unit_{[h_1,h_2]}^k) . \end{equation} 
By Proposition \ref{PROP2}, we have 
\begin{equation} \sum_{k=0}^n (-1)^k \log \det(\unit_{[h_1,h_2]}^k) +\log\left(\frac{T^{MS}_{(2)}(\mathcal D_2)}{T^{MS}_{(2)}(\mathcal D_1)}\right) = \sum_{p \in \Cr(f)} (-1)^{\ind(p)} \log \left( \det( h_1(p)^{-1} \circ h_2(p) ) \right). \end{equation}
For the remaining term, it is due to the main Theorem of \cite{Zhang:AN} that we have an equality
\begin{equation} \log \left(\frac{ T^{An}_{(2)}(\mathcal D_1)}{T^{An}_{(2)}(\mathcal D_2)}\right)  + \sum_{k=0}^n (-1)^k \log \det(\tau^k) = \log \left(\frac{ T^{RS}(\mathcal D_1)}{T^{RS}(\mathcal D_2)}\right). \end{equation}
Here, $T^{RS}(\mathcal D_i)$ denotes the {\itshape Ray-Singer Torsion element} as originally defined in \cite[Definition 2.2]{Bismut:Extension}. It is shown in \cite[Theorem 3.4]{Bruning:Glue} that, under the conditions that $M$ is odd-dimensional and either one of the two assertions mentioned in the statement of the proposition is satisfied, one has
\begin{equation}\label{ANO3} \log \left(\frac{ T^{RS}(\mathcal D_1)}{T^{RS}(\mathcal D_2)}\right) = 0. \end{equation}
The result direct follows from from \ref{ANO1} and \ref{ANO2}-\ref{ANO3}. 
\end{proof}

\begin{dfn}[Subdivision]  Let $M$ be a compact manifold and for $i=0,1$, let $(f_i,g_i)$ be a Morse-Smale pair. Then $(f_1,g_1)$ is called a {\bfseries subdivision} of $(f_0,g_0)$ if all of the following conditions are satisfied 
\begin{enumerate} \item $\Cr_p(f_0) \subseteq \Cr_p(f_1) \subseteq \bigcup_{x \in \Cr(f_0)} W^-_x(f_0)$ for each $0 \leq p \leq n$, 
\item $W^-_x(f_1) \subseteq W^-_x(f_0)$ for each $x \in \Cr(f_0)$, 
\item $W^-_x(f_0) = \bigcup_{y \in \Cr(f_1) \cap W^-_x(f_0)} W^-_y(f_1)$, and
\item $g_0 \equiv g_1$ near $\Cr(f_0) \cup \partial M$ and and $f_0 \equiv f_1$ near $\partial M$.
\end{enumerate}
\end{dfn}

We now describe the effect on the relative torsion under taking taking subdivisions. 
For that, let $M$ be a compact manifold, let $(f_i,g_i)$ be a Morse-Smale pair on $M$ for $i =0,1$, so that $(f_1,g_1)$ is a subdivision of $(f_0,g_0)$. Let $h$ be Hermitian form 
on a flat bundle $E \downarrow M$. By definition, there exists for each $y \in \Cr(f_1)$ a unique $x \in \Cr(f_0)$ satisfying $y \in W^-_x(f_0)$. Let $\widetilde{h}(y) \in \GL(E_y,\overline{E_y^*})$ be the Hermitian metric on $E_y$ obtained by parallel transport of the metric $h(x) \in \GL(E_x,\overline{E_x^*})$ along a curve connecting $x$ and $y$ that is entirely contained within $W^-_x(f_0)$. Note that since $W^-_x(f)$ is simply-connected, the resulting metric doesn't depend on the particular choice of curve. Note also that $\widetilde{h}(y) = h(y)$ whenever $h$ is a unitary metric. \\
For each $y \in \Cr(f_1)$, define \begin{equation} \omega(y) \coloneqq \log \det(\widetilde{h}(y)^{-1} \circ h(y) ) \in \reals_{\geq 0}. \end{equation} 
Observe that $\omega \equiv 0$ whenever $h$ is a unimodular metric. 
The proof of the following statement for closed manifolds is laid out in \cite[Proposition 5.3]{Friedlander:Rel} and carries over to general compact manifolds without further modification: 
\begin{proposition}\label{MSINV3}
In the above situation, we have 
\begin{equation} \mathcal R(E \downarrow M,g,h,\nabla_{g_0} f_0) - \mathcal R(E \downarrow M,g,h,\nabla_{g_1} f_1) = \sum_{y \in \Cr(f_1)} (-1)^{\ind (y)} \omega(y). \end{equation}
\end{proposition}

\begin{corollary}[Relative Torsion under subdivision] \label{RELINVSUB} Let $\mathcal D_0 = (E \downarrow M, g_0,h_0,\nabla_{g'_0}f_0)$ be a weakly admissible system with $M$ odd-dimensional and let $(f_1,g'_1)$ be a subdivision of $(f_0,g'_0)$. Then, one finds a Riemannian metric $g_1$ on $M$ and an Hermitian form $h_1$ with $g_1 \equiv g_0$ and $h_1 \equiv h_0$ near $\partial M$ on $E$, so that $\mathcal D_1 = (E \downarrow M, g_1,h_1, \nabla_{g'_1} f_1)$ is a weakly admissible system, satisfying
\begin{equation} \mathcal R(\mathcal D_0) = \mathcal R(\mathcal D_1). \end{equation} \end{corollary}

\begin{proof} For each $y \in \Cr(f_1)$, there exists by the definition of a subdivision a unique $x \in \Cr(f_0)$, such that $y \in W^-_x(f_0)$. As above, we let $\widetilde{h_1}(y) \in \GL(E_y,\overline{E_y}^*)$ be the Hermitian metric on the fiber $E_y$ obtained by parallel transport of the Hermitian metric $h_0(x) \in \GL(E_x,\overline{E_x^*})$ along a curve between $x$ and $y$ contained entirely within $W^-_x(f)$. With \begin{equation*} \omega(y) \coloneqq \log \det (\widetilde{h}_1(y)^{-1} \circ h_0(y)), \end{equation*} we obtain from Proposition \ref{MSINV3} \begin{equation}\label{morseinv2} \mathcal R(E \downarrow M, g_0,h_0,\nabla_{g_0'}f_0) = \mathcal R(E \downarrow M, g_0,h_0,\nabla_{g_1'}f_1) + \sum_{y \in \Cr(f_1)} (-1)^{\ind(y)} \omega(y). \end{equation} In order to construct the metric $h_1$, choose small disjoint open coordinate neighborhoods $U_y  \supset V_y \ni y$ around each $y \in \Cr(f_1)$, each also disjoint from from a neighborhood of the boundary, such that $\overline{V_y} \subset U_y$. Define the Hermitian form $h_1 \in \GL(E,\overline{E^*})$ to be an extension of the metrics $\bigcup_{y \in \Cr(f_1)} \widetilde{h}_1(y)$ that is parallel on $\bigcup_{y \in \Cr(f_1)} V_y$ and equal to $h_0$ on $M \setminus \bigcup_{y \in \Cr(f_1) \setminus \Cr(f_0)} U_y$. Lastly, choose a Riemannian metric satisfying $g_1 \equiv g_1'$ near $\Cr(f_1)$ and $g_1 \equiv g_0$ near $\partial M$ (in particular, $g_1$ is also of product form near $\partial M$).
By construction of the metrics $h_1$ and $g_1$, the system $\mathcal D_1 = (E \downarrow M, g_1,h_1, \nabla_{g'_1} f_1)$ is weakly admissible. Moreover, an application of Proposition \ref{METANOREAL} gives
\begin{equation}\label{aninv}  \mathcal R(E \downarrow M,g_0,h_0,\nabla_{g_1'} f_1) =  \mathcal R(E \downarrow M,g_1,h_1.\nabla_{g_1'}f_1) +  \sum_{y \in \Cr(f_1)} (-1)^{\ind(y)+1} \omega(y). \end{equation} 
The result now follows from \ref{morseinv2} and \ref{aninv}. 
\end{proof}

The proof of the last result of this section can be found in \cite[Proposition 3.7]{Friedlander:Bd}
\begin{proposition}[Determinant Class under Glueing]\label{glueglue} For $i =1,2$, let $(E_i \downarrow M_i)$ be two flat, complex bundles over a compact manifold, satisfying $E_1|_{\partial M_1} \downarrow \partial M_1 = E_2|_{\partial M_2} \downarrow \partial M_2$. Assume that both $E_i \downarrow M_i$ and $E_i|_{\partial M_i} \downarrow\partial M_i$ are of determinant class. Then, the flat bundle $E \downarrow M$ with $E \coloneqq E_1 \cup_{\partial E_1} E_2$ and $M \coloneqq M_1 \cup_{\partial M_1} M_2$ is of determinant class as well.  \end{proposition}

\section{Witten deformation and asymptotic expansions}
This section collects the main technical results achieved by Burghelea et al.\ in \cite{Friedlander:Bd} that are detrimental for the proof of Theorem \ref{MAINTHEOREM}. Since the methods employed by the authors carry over seamlessly from the unitary case to the general case of flat bundles, the proofs won't be included here.
\subsection{Witten deformation}
Throughout this section, we fix a countable group $\Gamma$ and a $\Gamma$-invariant Morse-Smale system $\mathcal D = (E \downarrow M,g,h,\nabla_{g'}f)$. For any parameter $t \in \reals_{\geq 0}$, the {\itshape Witten-deformation} $d_t$ of the exterior derivative $d$ on $\Omega^*(M,E)$ is defined as
\begin{equation} d_t \coloneqq e^{-t f}  d e^{t f}   = d + tdf \wedge \colon \Omega^*(M,E) \to \Omega^{*+1}(M,E). \end{equation}
Observe that $d_t^2 = 0$ for any $t \in \reals_{\geq 0}$, which is why we can regard the pair $\Omega^*_t(M,E) \coloneqq (\Omega^*(M,E),d_t)$ as a cochain complex. 
In analogy with the case $t = 0$, let $\delta_t \colon \Omega^*(M,E,g,h) \to \Omega^{*-1}(M,E,g,h)$ be the formal adjoint of $d_t$ with respect to the inner product \ref{IP} on $\Omega^*_c(M,E)$ induced by $g$ and $h$ and define \begin{align} & \Delta_{*,t} \coloneqq \delta^{*+1}_t d^*_t + d^{*-1}_t \delta^{*}_t: \dom(\Delta_{*,t}) \to \dom(\Delta_{*,t}), \\& \dom(\Delta_{*,t}) \coloneqq \{ \sigma \in \Omega^*_c(M,E) \colon \vec{n}\sigma = \vec{n}d_t^*\sigma = 0 \}. \end{align}
Observe that for any $t  \geq 0$, $d_t$ and $\delta_t$ are all elliptic differential operators of order $1$, while $\Delta_{*,t}$ is an elliptic differential operator of order $2$ that is symmetric (on its domain) with respect to the inner product on $\Omega^*_t(M,E)$ induced by $g$ and $h$. Moreover, just as in the case $t = 0$, one verifies that all three operators are closable when regarded as unbounded operators on the $L^2$-completion $\Omega^*_{(2)}(M,E)$. The closed, symmetric operator $\Delta_{*,t} \colon \Omega_{(2),t}^*(M,E) \to \Omega_{(2),t}^*(M,E)$ is called the {\itshape Witten-Laplacian} (with absolute boundary conditions) associated to the system $\mathcal D = (E \downarrow M,g,h,\nabla_{g'}' f)$. 
We define the {\itshape $L^2$-Witten-de Rham complex} of the system $\mathcal D$ as \begin{equation} \Omega^*_{(2),t}(M,E) \coloneqq (\Omega^*_{(2)}(M,E),d_t), \end{equation} where we identify $d_t$ with its minimal $L^2$-closure. 
All complexes obtained this way have the same isomorphism type. Namely, one easily sees that for each $t > 0$, multiplying a form $\omega$ with the function $e^{tf}$ determines an isomorphism of Hilbert $\vnN(\Gamma)$-cochain complexes \begin{align}& e^{tf} \colon \Omega^*_{(2),t}(M,E) \to \Omega_{(2)}^*(M,E) \label{driso}. \end{align} Furthermore, since $\mathcal D$ is a $\Gamma$-invariant system, it follows that $\Delta_{*,t}$ is the lift of an elliptic operator defined over a bundle on a compact manifold. With this in mind, one verifies as in the case $t = 0$ that $\Delta_{*,t}$ is in fact self-adjoint (with the imposed absolute boundary conditions).
In particular, for each $t \geq 0$, we can define the spectral projections \begin{equation} P^*(t) \coloneqq \chi_{[0,1)}(\Delta_{*,t}) \colon \Omega_{(2),t}^*(M,E) \to \Omega_{(2),t}^*(M,E), \end{equation} of $\Delta_{*,t}$ associated wit the half-open interval $[0,1)$, as well as the {\itshape small} and {\itshape large} subcomplexes \begin{align}& \Omega^*_{Sm,t}(M,E) \coloneqq \left(\bigoplus_{k=0}^n \im(P^k(t)),d_t\right), \\&  \Omega^*_{La,t}(M,E) \coloneqq \left(\bigoplus_{k=0}^n \im (\unit_{\Omega^k_{(2)}(M,E)} - P^k(t)),d_t\right). \end{align}
Because $\Delta_{*,t}$ commutes with its spectral projections, one verifies inductively that $\im(P^*(t)) \subseteq \dom(\Delta_{*,t}^l)$ for each $l \in \mathbb N_0$. Together with the ellipticity of $\Delta_{*,t}$, we deduce \begin{equation} \label{INCL}\Omega^*_{Sm,t}(M,E) \subseteq \bigcap_{l=0}^\infty \mathcal W_{l}^*(M,E). \end{equation} In particular, the complex $\Omega^*_{Sm,t}(M,E)$ consists entirely of smooth forms. Moreover, observe that we have an orthogonal decomposition of Hilbert $\vnN(\Gamma)$-cochain complexes \begin{equation} \Omega^*_{(2),t}(M,E) = \Omega^*_{Sm,t}(M,E) \oplus \Omega^*_{La,t}(M,E). \end{equation} 
Finally, just as in the case $\partial \widetilde{M} = \emptyset$, one verifies: 
\begin{proposition}\cite[Theorem 4.2]{Shubin:coeff}\label{shubinch} For each $t \geq 0$, the projection $P^*(t) \colon \Omega^*_{(2),t}(M,E) \to \Omega^*_{Sm,t}(M,E)$ onto the small subcomplex is a chain homotopy equivalence of Hilbert $\vnN(\Gamma)$-cochain complexes (with chain homotopy inverse given by the inclusion). \end{proposition}

Now assume additionally that the $\Gamma$-invariant system $\mathcal D =(E \downarrow M,g,h,\nabla_{g'} f)$ is also weakly $\Gamma$-admissible. Recall from the axioms laid out in Definition \ref{ADMISSIBLESYSTEM} that a $\Gamma$-invariant system is weakly $\Gamma$-admissible if certain local conditions near $\Cr(f)$ are satisfied: We can choose for each $p \in \Cr(f)$ radii $r_p > 0$, coordinate charts $\phi_p : B_{r_p}(0) \stackrel{\cong}{\to} U_p \subseteq \reals^n$ disjoint from $\partial M$ with $B_{r_p}(0) \coloneqq \{x \in \reals^n : ||x|| < r_p \}$ and $\phi_p(0) = p$, along with a flat bundle isomorphism $\Phi_p: B_{r_p}(0) \times \ceals^m \stackrel{\cong}{\to}  E|_{U_p}$ that fit into the commutative diagram \begin{equation} \begin{tikzcd}B_{r_p}(0) \times \ceals^m  \ar{r}{\Phi_p}[swap]{\cong} \ar{d}{pr_1} &  E|_{U_p} \ar{d}{\pi_E} \\ B_{r_p}(0) \ar{r}{\phi_p}[swap]{\cong} & U_p , \end{tikzcd} \end{equation} and such that all of the following conditions hold:
 \begin{enumerate}[label=(\subscript{H}{{\arabic*}})]
\item The pullback metric $\phi_p^*(g|_{U_p})$ equals the Euclidean metric on $\reals^n$. 
\item The pullback Hermitian form $\Phi_p^*(h|_{U_p})$ equals the standard inner product on $\ceals^m$. 
\item One has \begin{equation*} (f \circ \phi_p)(x_1,\dots,x_n) = f(p) - \frac{1}{2}\sum_{i=1}^{\ind(p)} x_i^2 + \frac{1}{2} \sum_{i=\ind(p)+1}^{n} x_i^2. \end{equation*}
\item The above choices are $\Gamma$-invariant, i.e.\ $\gamma.U_p = U_{\gamma.p}$, $r_{p} = r_{\gamma.p}$, $\gamma \circ \phi_p = \phi_{\gamma.p}$ and $\gamma \circ \Phi_p = \Phi_{\gamma.p}$ for each $p \in \Cr(f)$ and each $\gamma \in \Gamma$. 
\end{enumerate}
It is precisely due to this $\Gamma$-invariant shape of $f$ and metric bundle $(E,h) \downarrow (M,g)$ near $\Cr(f)$ that Burghelea et al.\ were able to prove the next theorem. With the aid of properties $(H_1)-(H_4)$, their proof from \cite[Section 3.3]{Friedlander:Bd} can be adapted, word by word, to our situation of non-unitary bundles without any further modification:
\begin{theorem} Let $(E \downarrow M,g,h,\nabla_{g'} f)$ be a weakly $\Gamma$-admissible system. Then, for each $t \geq 0$, there exists an isometric embedding of Hilbert $\vnN(\Gamma)$-modules 
\begin{equation*} J^*(t) \coloneqq \bigoplus_{k = 0}^n J^k(t): C_{(2)}^{\bu}(M,\nabla_{g'}f,E,h) \to \Omega_{(2)}^{\bu}(M,E), \end{equation*} Moreover, for large $t >> 0$, the composition \begin{equation*} Q(t) \coloneqq P^*(t) \circ J^*(t):  C_{(2)}^{\bu}(M,\nabla_{g'}f,E,h) \to \Omega_{Sm,t}^{\bu}(M,E) \end{equation*}
is an isomorphism of Hilbert $\vnN(\Gamma)$-modules. \end{theorem}

We stress the fact that the map of Hilbert $\vnN(\Gamma)$-modules $J^*(t)$ from the previous theorem (and therefore also the isomorphism $Q^*(t)$) is in general {\bfseries not} a map of {\bfseries cochain complexes}. This is why the maps $Q^*(t)$ alone cannot be used to reach our desired conclusion, namely that the complexes $C_{(2)}^{\bu}(M,\nabla_{g'}f,E,h)$ and $\Omega_{Sm,t}^{\bu}(M,E,g,h)$ are chain homotopy equivalent. 
In spite of this, it still follows that for sufficiently large $t >> 0$, the isomorphism $Q^*(t)$ can be used to define the isometry \begin{equation}\label{ISOMETRY} I^*(t) \coloneqq Q^*(t) \left( Q^*(t)^* Q^*(t) \right)^{-1/2}:  C_{(2)}^* (M,\nabla_{g'}f,E,h) \to \Omega_{Sm,t}^*(M,E,g,h). \end{equation}
Moreover, since $(E \downarrow M,g,h,\nabla_{g'} f)$ is the lift of an admissible system with deck group $\Gamma$, there are also isomorphisms of Hilbert $\vnN(\Gamma)$-modules for $t > 0$: \begin{align}\label{SCALAR}& S^*(t): C^*_{(2)}(M,\nabla_{g'}f,E,h) \to C^*_{(2)}(M,\nabla_{g'}f,E,h),
\\& \lambda_p \otimes [p] \mapsto  \exp \left( \frac{n-2\ind(p)}{4}\log(\pi/t)-tf(p) \right) \cdot \lambda_p \otimes [p],\hspace{.5cm} p \in \Cr(f). \end{align}
Here, we have used the fact that $f$ is $\Gamma$-invariant, hence in particular satisfies $f(\gamma.x) = f(x)$ for any $x \in M$. \\ Recall that because of Equation \ref{INCL}, we have the inclusion $\Omega^*_{Sm,t}(M,E) \subseteq \bigcap_{l \in \mathbb N} \mathcal W^*_l(M,E)$. This allows us to define the morphism of Hilbert $\vnN(\Gamma)$-cochain complexes  \begin{equation} F^*(t) \coloneqq \Int^* \circ e^{tf}: \Omega_{Sm,t}^*(M,E,g,h) \to C^*_{(2)}(M,\nabla_{g'}f,E,h), \end{equation} as restricting to the subcomplex $\Omega^*_{Sm,t}(M,E)$ the composition of the isomorphism \begin{equation*} e^{tf}: \Omega_{(2),t}^*(M,E) \to  \Omega_{(2)}^*(M,E) \end{equation*} from \ref{driso} with the integration map \begin{equation*} \Int^*: \mathcal W_{l - *}^*(M,E) \to C^*_{(2)}(M,\nabla_{g'}f,E,h), \end{equation*} defined as in \ref{DRINT}. Just as before, the proof of the next theorem, laid out for unitary bundles in \cite[Section 3.3]{Friedlander:Bd}, can be adapted to our setting without any modifications:

\begin{theorem}\label{MSSMEQ} Under the previous assumptions, we obtain for large $t >> 0$, that \begin{equation} S^*(t)  \circ F^*(t) \circ I^*(t) = \unit + \mathcal O(t^{-1}). \end{equation} 
Consequently, for large $t >> 0$, the map $F^*(t) \colon \Omega_{Sm,t}^*(M,E,g,h) \to C^*_{(2)}(M,\nabla_{g'}f,E,h) $ is an isomorphism of Hilbert $\vnN(\Gamma)$-cochain complexes. \end{theorem} 

Combining \ref{driso} with Proposition \ref{shubinch} and Theorem \ref{MSSMEQ}, we arrive at the following very important intermediate result:

\begin{theorem}\label{ANCOMBCH} Let $\mathcal D = (M,E,g,h,\nabla_{g'}f)$ be a weakly admissible type \rom{2}-Morse Smale system with $M$ compact, let $\widetilde{M}$ be the universal cover of $M$ and let $\widetilde{\mathcal D} = (\widetilde{M},\widetilde{E},\widetilde{g},\widetilde{h},\nabla_{\widetilde{g'}}\widetilde{f})$ be the corresponding lift of $\mathcal D$. Then, there is a chain homotopy equivalence of Hilbert $\vnN(\Gamma)$-cochain complexes
\begin{equation} \begin{tikzcd} \Omega_{(2)}^*(\widetilde{M},\widetilde{E},\widetilde{g},\widetilde{h}) \ar{d}{e^{-t\widetilde{f}}}[swap]{\ref{driso}} \ar{r}{\simeq} & C^*_{(2)}(\widetilde{M},\nabla_{\widetilde{g'}}\widetilde{f},\widetilde{E},\widetilde{h}) \\ \Omega_{(2),t}^*(\widetilde{M},\widetilde{E},\widetilde{g},\widetilde{h}) \ar{r}{P^*(t)}[swap]{\ref{shubinch}} & \Omega_{Sm,t}^*(\widetilde{M},\widetilde{E},\widetilde{g},\widetilde{h}) \ar{u}{F^*(t)}[swap]{\ref{MSSMEQ}}\end{tikzcd},\end{equation} with $t>>0$ chosen sufficiently large.
In particular, we obtain:
\begin{enumerate}
\item For each $0 \leq k \leq n$, it holds that $\alpha_k^{An}(M,E) = \alpha_k^{Top}(M,E)$.
\item $(M,E)$ is of a-determinant class if and only if it is of c-determinant class.
\end{enumerate}
\end{theorem}

\subsection{Asymptotic expansions}

Let $\mathcal D = (E \downarrow M,g,h,\nabla_{g}f)$ be a weakly admissible system with $\Gamma \coloneqq \pi_1(M)$ and let $\widetilde{\mathcal D} \coloneqq (\widetilde{E} \downarrow \widetilde{M},\widetilde{g},\widetilde{h},\nabla_{\widetilde{g}}\widetilde{f})$ be the $\Gamma$-invariant lift of $\mathcal D$ (throughout this subsection, we assume that $g = g'$). We set $b \coloneqq f^{-1}(\partial M)$. 
For $t \geq 0$, let $\Omega_{(2),t}^*(\widetilde{M},\widetilde{E})$ be the Witten-deformed complex defined in the previous section (with metric induced by $\widetilde{g}$ and $\widetilde{h}$ implicit, in order to simplify notation) with Witten-deformed Laplacian $\Delta_{*,t}[\widetilde{E}]: \Omega_{(2),t}^*(\widetilde{M},\widetilde{E}) \to \Omega_{(2),t}^*(\widetilde{M},\widetilde{E})$.
the orthogonal decomposition into the small and large subcomplex. Further, we define $\Theta^*(t): \ker(\Delta_{*,t}[E]) \to H_{(2)}^*(\widetilde{M},\nabla_{\widetilde{g}}\widetilde{f},\widetilde{E},\widetilde{h})$ to be the isomorphim of finitely-generated Hilbert $\vnN(\Gamma)$-modules that is the
composition $\Theta^* \cdot e^{t\widetilde{f}}$, where $\Theta^*: \ker(\Delta_{*,0}[E]) \to  H_{(2)}^*(\widetilde{M},\nabla_{\widetilde{g}}\widetilde{f},\widetilde{E},\widetilde{h})$ is the isomorphism from \ref{MSINTG}. 
Introduce \begin{align}  \text{Vol}(\mathcal D)(t) \coloneqq \prod_{k=0}^n \vndet(\Theta^k(t))^{(-1)^{k}}. \end{align} Observe that \begin{equation} \text{Vol}(\mathcal D)(0) = T_{(2)}^{Met}(E \downarrow M,g,h,\nabla_{g}f). \end{equation}
Moreover, recall the orthogonal decomposition of subcomplexes 
\begin{align*} \Omega_{(2),t}^*(\widetilde{M},\widetilde{E}) = \Omega_{Sm,t}^*(\widetilde{M},\widetilde{E}) \oplus \Omega_{La,t}^*(\widetilde{M},\widetilde{E}), \end{align*} which implies the following: 
Provided that $E \downarrow M$ is of determinant class, the torsion elements $T^{An}_{(2)}(E \downarrow M,g,h)(t), T^{Sm}_{(2)}(\mathcal D)(t)$ and $T^{La}_{(2)}(\mathcal D)(t)$ of the complexes $\Omega_{(2),t}^*(\widetilde{M},\widetilde{E})$, $\Omega_{Sm,t}^*(\widetilde{M},\widetilde{E})$, respectively $\Omega_{La,t}^*(\widetilde{M},\widetilde{E})$ are
all well-defined positive real numbers, so that
\begin{align} T^{An}_{(2)}(E \downarrow M,g,h)(0) = T^{An}_{(2)}(M,E,g,h),\\
T^{An}_{(2)}(E \downarrow M,g,h)(t) = T^{Sm}_{(2)}(\mathcal D)(t)\cdot T^{La}_{(2)}(\mathcal D)(t). \end{align}
A function $F: \reals \to \reals$ is said to admit an {\itshape asymptotic expansion}, if there exists an integer $N \in \mathbb N$ and constants $(a_j)_{j = 0}^N , (b_j)_{j = 0}^N$ such that for $t \to +\infty$
\begin{equation} F(t) = \sum_{j=0}^N \left( a_j + b_j \log(t) \right) t^j + o(1). \end{equation} The coefficient $a_0$ in the expansion is called the {\itshape free term} of $F$ and is denoted by $\text{FT}(F)$. 
In the special case that $E \downarrow M$ is a unitary bundle and $h$ a flat unitary metric, the proof of the next proposition has been carried out in \cite[Theorem 3.13]{Friedlander:Bd}. In fact, the same proof still works without further modification in the more general case that the bundle $E \downarrow M$ is of product from near $\partial M$, and will therefore be omitted (See also \cite[Proposition 6.4.1]{Ich} for a slightly different proof).

\begin{proposition}[Asymptotic expansion for the analytic torsion] There exists a constant $C \in \reals$, such that the following holds:  For any weakly admissible system $\mathcal D = (E \downarrow M,g,h,\nabla_{g}f)$ of determinant class with $M$ odd-dimensional, the function $\log T^{An}_{(2)}(E \downarrow M,g,h)(t) - \log \text{Vol}(\mathcal D)(t)$ admits the asymptotic expansion. 
\begin{align}\label{ASYANA}   \log T^{An}_{(2)} (E \downarrow M,g,h) -  \log \text{Vol}(\mathcal D)(0) + C t  \dim(E)  \chi(\partial M) . \end{align}
\end{proposition}

\begin{proposition}[Asymptotic expansion for the small torsion] For any weakly admissible system $\mathcal D = (E \downarrow M,g,h,\nabla_{g} f)$ of determinant class with $n \coloneqq \dim(M)$, $\Cr_k(f) \coloneqq \{p \in \Cr(f) \colon \ind(p) = k\}$ and $m_k \coloneqq \# \Cr_k(f)$, the function $\log T^{Sm}_{(2)}(\mathcal D)(t) - \log \text{Vol}(\mathcal D)(t)$ admits the asymptotic expansion
\begin{align}&\label{ASYSMALL} \log T^{(2)}_{MS}(M,E,h,\nabla_{g}f)  + \dim(E) \left( \sum_{k=0}^n (-1)^{k} m_k  \frac{n - 2k}{4}  \log(\pi/t)   +t   (-1)^{k+1}   \sum_{p \in \Cr_k(f)} f(p) \right) + o(1). \end{align}
\end{proposition}
\begin{proof} For large $t >> 0$, there exists by Theorem \ref{MSSMEQ} an isomorphism of finitely generated Hilbert $\vnN(\Gamma)$-cochain complexes 
\begin{equation} F^*(t): \Omega_{Sm,t}^*(\widetilde{M},\widetilde{E},\widetilde{g},\widetilde{h}) \to C^*_{(2)}(\widetilde{M},\nabla_{\widetilde{g}}\widetilde{f},\widetilde{E},\widetilde{h}). \end{equation} 
From Proposition \ref{PROP2}, it then follows that
\begin{equation}\label{ASYSM1} \log T^{Sm}_{(2)}(M,E,g,h,f)(t) - \log \text{Vol}(t) = \log T^{(2)}_{MS}(M,E,h,\nabla_{g}f)  - \sum_{k=0}^n (-1)^k \log \vndet F^k(t). \end{equation} 
Recall also from Theorem \ref{MSSMEQ} the formula $S^k(t) \circ F^k(t) \circ I^k(t) = \unit_{C^k_{(2)}} + O(t^{-1})$, where $I^k(t)$ is the isometry from \ref{ISOMETRY} and $S^k(t)$ is the scaling isomorphism from \ref{SCALAR}. 
Consequently, by the multiplicativity of the Fuglede-Kadison determinant in this setting \cite[Theorem 3.14]{Lueck:book}, it holds that
\begin{equation} \log \vndet F^k(t) = - \log \vndet S^k(t) + o(1). \end{equation} 
From the explicit formula of $S^k(t)$ \ref{SCALAR}, we obtain 
\begin{equation}\label{ASYSM2} \vndet S^k(t) = \left( \prod_{p \in \Cr_k(f)} (\pi/t)^{\frac{n-2k}{4}}e^{-tf(p)} \right)^{\dim(E)}. \end{equation}
The result now is an immediate consequence of \ref{ASYSM1} -- \ref{ASYSM2}. 
\end{proof}

\begin{corollary}[Asymptotic expansion for the large torsion]\label{rella} Let $\mathcal D = (E \downarrow M, g,h,\nabla_{g} f)$ be a weakly admissible system of determinant class with $M$ odd-dimensional. Then, the following assertions hold
\begin{enumerate} \item The function $\log T^{La}_{(2)}(\mathcal D)(t)$ admits an asymptotic expansion. More precisely, there exists a polynomial $\Phi(\mathcal D)(t): \reals \to \reals$ in $t$ and $\log(t)$, such that for $t \to \infty$ 
\begin{equation} \log T^{La}_{(2)}(\mathcal D)(t) = R(\mathcal D) + \Phi(\mathcal D)(t) + o(1). \end{equation}
Finally, for any arbitrary small neighborhood $U$ of $\Cr(f) \cup \partial M$, the polynomial $\Phi(\mathcal D)$ depends only on the isomorphism class of the system $\mathcal D^f|_{U} \coloneqq (E|_{U} \downarrow U, g|_{U},h|_{U},f|_{U})$.
\item Suppose that $\mathcal D_1 = (E_1 \downarrow M_1,g_1,h_1,\nabla_{g_1} f_1)$ is another weakly admissible system, such that there exists neighborhoods $U \subseteq M$ of $\Cr(f) \cup \partial M$ and $U_1 \subset M_1$ of $\Cr(f_1) \cup \partial M_1$ with the property that the derived systems $\mathcal D^f|_U \coloneqq (E|_{U} \downarrow U, g|_{U},h|_{U},f|_{U})$
and \\$\mathcal D^{f_1}_1|_{U_1} \coloneqq (E_1|_{U_1} \downarrow U_1, g|_{U_1},h|_{U_1},f_1|_{U_1})$ are isomorphic (in particular $\# \Cr_k(f) = \# \Cr_k(f_1)$ for each $0 \leq k \leq n$).
Then \begin{align} R(\mathcal D) - R(\mathcal D_1) = \text{FT}\left( \log T^{La}_{(2)}(\mathcal D) \right) - \text{FT}\left( \log T^{La}_{(2)}(\mathcal D_1) \right). \end{align}
\item Under the assumptions of $(2)$, there exists local quantities $\alpha(\mathcal D) \in \Omega^n(M\setminus \Cr(f),\mathcal O_M)$ and $\alpha(\mathcal D_1) \in \Omega^n(M_1\setminus \Cr(f_1),\mathcal O_{M_1})$ of the derived systems systems $\mathcal D|_{M \setminus \Cr(f)} $ and $\mathcal D_1|_{M_1 \setminus \Cr(f_1)}$, such that one has
\begin{equation} \text{FT}\left( \log T^{La}_{(2)}(\mathcal D) \right) - \text{FT}\left( \log T^{La}_{(2)}(\mathcal D_1) \right)  = \int_{M \setminus \Cr(f)} \alpha(\mathcal D) - \int_{M \setminus \Cr(f_1)} \alpha(\mathcal D_1).\end{equation}
\end{enumerate}
\end{corollary} 

\begin{proof} 
\begin{enumerate} \item We have $\log T^{An}_{(2)}(E \downarrow M,g,h)(t) = \log T^{Sm}_{(2)}(\mathcal D)(t) + \log T^{La}_{(2)}(\mathcal D)(t)$, hence also in particular
\begin{equation*} \left( \log T^{An}_{(2)}(E \downarrow M,g,h)(t) - \log \text{Vol}(\mathcal D)(t) \right) -  \left( \log T^{Sm}_{(2)}(\mathcal D)(t) - \log \text{Vol}(\mathcal D)(t) \right) = \log T^{La}_{(2)}(\mathcal D)(t). \end{equation*}
Since the left-hand side of the equation admits an asymptotic expansion, given by the sum of the explicit formulas \ref{ASYANA} and \ref{ASYSMALL}, the result follows. 
\item Observe that $\text{FT}\left( \log T^{La}_{(2)}(\mathcal D) \right) = R(\mathcal D) +  \text{FT}(\Phi(\mathcal D))$ and analogously $\text{FT}\left( \log T^{La}_{(2)}(\mathcal D_1) \right) = R(\mathcal D_1) +  \text{FT}(\Phi(\mathcal D_1))$. 
Since the systems $\mathcal D^f|_{U}$ and $\mathcal D^{f_1}|_{U_1}$ are isomorphic by assumption, assertion $(1)$ implies that $\Phi(\mathcal D) \equiv \Phi(\mathcal D_1)$ and the result follows. 
\item In case that $\partial M = \emptyset$, this is proven in \cite[Theorem B, Section 6.2]{Friedlander:Uni} for unitary bundles (whose proof is also referred to in \cite[Proposition 4.2]{Friedlander:Rel} for arbitrary flat bundles). The same proof works without any modifications in the case that $\partial M \neq \emptyset$. 
\end{enumerate}
\end{proof}

\section{Proof of Theorem \ref{MAINTHEOREM}}
Armed with the results of the previous two sections, we will closely follow the strategy of \cite{Friedlander:Bd} and use Zhang's result \ref{CLOSINV2} to prove Theorem \ref{MAINTHEOREM}. 
\begin{proposition}\label{gluecomp} For $i=1,2$, let $\mathcal D_i = (M_i,E_i,g_i,h_i,\nabla_{g_i} f_i)$ be two weakly admissible systems satisfying the assumptions of Corollary \ref{rella}.2. Moreover, assume that there exists a flat bundle $E_3 \downarrow M_3$ with $M_3$ compact, satisfying
\begin{enumerate} \item $(E_3|_{\partial M_3}) \downarrow \partial M_3 = E_i|_{\partial M_i} \downarrow \partial M_i$, and
\item the bundle $\overline{E_i} \downarrow N_i$ is of determinant class, where $N_i \coloneqq M_3 \cup_{\partial M_3} M_i$ and \\ $\overline{E_i}\coloneqq E_3 \cup_{E_3|_{\partial M_3}} E_i$. 
\end{enumerate}
Then \begin{equation}\label{relcom1} \mathcal R(\mathcal D_1) + \frac{1}{2}\int_{M_1} \theta(h_1) \wedge (\nabla_{g_1}f_1)^* \Psi(TM_1,g_1) = \mathcal R(\mathcal D_2) + \frac{1}{2}\int_{M_2} \theta(h_2) \wedge (\nabla_{g_2}f_2)^* \Psi (TM_2,g_2). \end{equation}
\end{proposition}

\begin{proof} Choose a smooth function $f_3: M_3 \to \reals$ on $M_3$ with $f_3|_{\partial M_3} = f_i|_{\partial M_i}$ for $i=1,2$ and such that the function $\overline{f}_i \coloneqq f_3 \cup_{\partial M_3} f_i: N_i \to \reals$ is a Morse function. Furthermore, choose a Riemannian metric $g_3$ on $M_3$ with $g_3|_{\partial M_3} = g_i|_{\partial M_i}$ for $i=1,2$, such that for the metric $\overline{g}_i  = g_3 \cup_{\partial M_3} g_i$ on $N_i$, the pair $(\overline{f}_i,\overline{g}_i)$ is a Morse-Smale pair (since $N_i$ is closed, there is no distinction between type \rom{1} and type \rom{2}). Lastly, choose a Hermitian form $h_3$ on the flat bundle $E_3 \downarrow M_3$ with $h_3|_{\partial M_3} = h_i|_{\partial M_i}$ for $i=1,2$ with $\overline{h}_i \coloneqq h_3 \cup_{\partial M_3} h_i$, such that the system 
\begin{equation} \overline{\mathcal {D}}_i \coloneqq (\overline{E_i} \downarrow N_i,\overline{g}_i,\overline{h}_i,\nabla_{\overline{g}_i}\overline{f}_i)\end{equation} is weakly admissible. By construction, the pair $\overline{\mathcal{D}_i}$ also satisfies the assumptions of Corollary \ref{rella}.2. Applying Corollary \ref{rella}.3, we can find densities $\alpha_i $ on $M_i \setminus \Cr(f_i)$ and $\overline{\alpha_i}$ on $N_i \setminus \Cr(\overline{f_i})$, so that
\begin{align} \label{dif1} R(\mathcal D_1) - R(\mathcal D_2)  = \int_{M_1 \setminus \Cr(f_1)} \alpha_1 - \int_{M_2 \setminus \Cr(f_2)} \alpha_2, \\
\label{dif2} R(\overline{\mathcal D_1}) - R(\overline{\mathcal D_2})  = \int_{N_1 \setminus \Cr(\overline{f_1})} \overline{\alpha_1} - \int_{N_2 \setminus \Cr(\overline{f_2})} \overline{\alpha_2}. \end{align}
Since the densities are local quantities, it follows from the chosen metrics on the respective bundles that $\alpha_i = \overline{\alpha_i}|_{M_i}$ and $\overline{\alpha_1}|_{M_3} = \overline{\alpha_2}|_{M_3}$. Moreover, since $\Cr(\overline{f_i}) \cap M_i = \Cr(f_i)$ by construction, we get from \ref{dif1} and \ref{dif2}
\begin{align}\label{bdclos} R(\mathcal D_1) - R(\mathcal D_2) = R(\overline{\mathcal D}_1) - R(\overline{\mathcal D}_2).\end{align}
 As $N_i$ is closed, we can apply Theorem \ref{CLOSINV2} and obtain 
\begin{align}  & R(\overline{\mathcal D_i}) = \frac{1}{2}\int_{N_i} \theta(\overline{E_i},\overline{h_i} ) \wedge (\nabla_{\overline{g_i}} \overline f_i)^*\Psi (TN_i,\overline{g_i}) \label{clos}, \end{align}
As mentioned in the introduction, the $n$-form $\theta(\overline{E_i},\overline{h_i}) \wedge (\nabla_{\overline{g_i}} \overline f_i)^*\Psi (TN_i,\overline{g_i})$ is a local quantity. In particular, it follows both that $\theta(\overline{E_i},\overline{h_i}) \wedge (\nabla_{\overline{g_i}} \overline {f_i})^*\Psi (TN_i,\overline{g_i})|_{M_i} = \theta(E_i,h_i) \wedge (\nabla_{g_i}f_i)^*\Psi(TM_i,g_i)$ and that $\theta(\overline{E_1},\overline{h_1} )\wedge(\nabla_{\overline{g_1}} \overline f_1)^*\Psi (TN_1,\overline{g_1})|_{M_3} = \theta(\overline{E_2},\overline{h_2} )\wedge(\nabla_{\overline{g_2}} \overline f_2)^*\Psi (TN_2,\overline{g_2})|_{M_3}$. Therefore \begin{align}& \int_{N_1} \theta(\overline{E_1},\overline{h_1} )\wedge(\nabla_{\overline{g_1}} \overline f_1)^*\Psi (TN_1,\overline{g_1}) \nonumber - \int_{N_2} \theta(\overline{E_2},\overline{h_2} )\wedge(\nabla_{\overline{g_2}} \overline f_2)^*\Psi (TN_2,\overline{g_2}) \nonumber\\ \label{loc2}& = \int_{M_1} \theta(E_1,h_1)\wedge(\nabla_{g_1}f_1)^* \Psi(TM_1,g_1) - \int_{M_2} \theta(E_2,h_2)\wedge(\nabla_{g_2}f_2)^* \Psi (TM_2,g_2).\end{align} Equation \ref{relcom1} now is an immediate consequence of \ref{bdclos} -- \ref{loc2}.  
\end{proof}

\begin{theorem}\label{strongthm} Assume that $\mathcal D_i = (E_i \downarrow M_i,g_i,h_i,\nabla_{g_i'}f_i)$ are two admissible systems with $M_i$ odd-dimensional, $(\partial M_1, g_1|_{\partial M_1}) = (\partial M_2, g_2|_{\partial M_2})$ and  $(E_1|_{\partial M_1},h_1|_{\partial M_1}) = (E_2|_{\partial M_2},h_2|_{\partial M_2})$. Then, if both $E_i \downarrow M_i$ and $E_i|_{\partial M_i} \downarrow M_i$ are of determinant class, we get
\begin{equation}\label{power} \mathcal R(\mathcal D_1) + \frac{1}{2} \int_{M_1} \theta(E_1,h_1) \wedge  (\nabla_{g_1'}f_1)^* \Psi(TM_1,g_1) \nonumber \\ = \mathcal R(\mathcal D_2) + \frac{1}{2}\int_{M_2} \theta(E_2,h_2) \wedge (\nabla_{g_2'}f_2)^* \Psi(TM_2,g_2). \end{equation}
\end{theorem}
\begin{proof}
We consider different cases: \\
{\bfseries Case 1: The systems $\mathcal D_i$ satisfy the hypotheses of Corollary \ref{rella}.2:} \\
Consider the admissible system $\mathcal D_{S^2}: (E^{S^2}_{\ceals} \downarrow S^2,g,h,\nabla_{g} f)$ with $E^{S^2}_{\ceals} \downarrow S^2$ the the trivial complex line bundle over $S^2$, $(f,g)$ some Morse-Smale pair on $S^2$ and $h$ a parallel metric on $E_{\ceals}^{S^2}$. Since $S^2$ is simply-connected, the system $\mathcal D_{S^2}$ is of determinant class. It follows from Proposition \ref{PRODFORM2} that that also the modified product systems $\underline{\mathcal D_i \times \mathcal D_{S^2}}$  are of determinant class, so that
\begin{align}\label{PROD} \mathcal R(\underline{\mathcal D_i \times \mathcal D_{S^2}}) = 2 \mathcal R(\mathcal D_i), \end{align} where we have used that $\chi(S^2) = 2$, as well as the well-known fact that $\mathcal R(\mathcal D_{S^2}) = 0$, which follows for example also from Theorem \ref{CLOSINV2}. \\
Next, consider the trivial complex line bundle $E^{D^3}_{\ceals} \downarrow D^3$. Since $D^3$ is simply-connected, it is of determinant class. Moreover, since $E|_{\partial M_1} \downarrow \partial M_1$ is of determinant class by assumption and $\partial M_1$ is closed, it follows again from Proposition \ref{PRODFORM2} that the product bundle $E|_{\partial M_1} \hat \otimes E^{D^3}_{\ceals} \downarrow \partial M_1 \times D^3$, as well as its restriction to $\partial (\partial M_1 \times D^3) = \partial M_1 \times \partial D^3$, is of determinant class. Now observe that by construction, the identification $\partial D^3 \cong S^2$ induces an isomorphism of flat bundles $E_1 \hat \otimes E^{D^3}_{\ceals}|_{\partial M_1 \times \partial D^3} \downarrow \partial M_1 \times \partial D^3 \cong E_i \hat \otimes E^{S^2}_{\ceals}|_{\partial M_i \times S^2} \downarrow \partial M_i \times S^2$ for $i =1,2$. Just as in Proposition \ref{gluecomp}, we can therefore define for $i=1,2$
\begin{align*} &N_i : = M_i \times S_2 \cup_{\partial M_1 \times S^2} \partial M_1 \times D^3,\\& \overline{E_i} \coloneqq E_i \hat \otimes E^{S^2}_{\ceals} \cup_{E_i|_{\partial M_i} \hat \otimes E^{S^2}_{\ceals}} E_1 \hat \otimes E^{D^3}_{\ceals}. \end{align*}
By Proposition \ref{glueglue}, it follows that $\overline{E_i} \downarrow N_i$ is of determinant class. Hence, the modified product systems $\underline{\mathcal D_i \times \mathcal D_{S^2}}$ satisfy also the assumptions of Proposition \ref{gluecomp}, from which we get
\begin{align} & \mathcal R(\underline{\mathcal D_1 \times \mathcal D_{S^2}})  + \frac{1}{2}\int_{M_1 \times S^2} \theta(h_1 \hat \otimes h) \wedge \nabla_{g_1 \times g}(\underline{f_1 + f})^* \Psi\left(T(M_1 \times S^2), g_1 \times g\right)  \nonumber \\
\label{propgood} & = \mathcal R(\underline{\mathcal D_2 \times \mathcal D_{S^2}}) + \frac{1}{2}\int_{M_2 \times S^2} \theta(h_2 \hat \otimes h) \wedge (\nabla_{g_2 \times g}(\underline{f_2 + f}))^* \Psi\left(T(M_2 \times S^2), g_2 \times g\right).  \end{align}
 Applying the product formula \ref{fac1}, we obtain for $i =1,2$ 
\begin{align*}\theta(h_i \hat \otimes h) \wedge \nabla_{g_i \times g}(\underline{f_i + f})^* \Psi\left(T(M_i \times S^2), g_i \times g\right) = \left( \theta(h_i) \wedge (\nabla_{g_i}f_i)^* \Psi (TM_i,g_i) \right)  \otimes e(TS^2,g), \end{align*}
Since $e(TS^2,g)$ is a representative of the rational Euler class of $TS^2$, we obtain that $\int_{S^2} e(TS^2,g) = \chi(S^2) = 2$. Together with the previous equation, this  implies for $i=1,2$, that \begin{align}\label{prodeul} \int_{M_i \times S^2}\theta(h_i \hat \otimes h) \wedge \nabla_{g_i \times g}(\underline{f_i + f})^* \Psi\left(T(M_i \times S^2), g_i \times g\right) 
= 2 \int_{M_i} \theta(h_i) \wedge (\nabla_{g_i}f_i)^* \Psi (TM_i,g_i). \end{align}
The result now follows from \ref{PROD} -- \ref{prodeul}. \\
{\bfseries Case 2: The systems $\mathcal D_i$ don't satisfy the hypotheses of Corollary \ref{rella}.2:} \\
Since the $\mathcal D_i$ are by assumption admissible, we find a neighborhood $U$ of $\partial M$, such $\theta(h_i) \equiv 0$ on $U$ and $g_i \equiv g_i'$ on $M \setminus U$, which is why $\theta(h_i) \wedge  (\nabla_{g_i'}f_i)^* \Psi(TM_i,g_i) = \theta(h_i) \wedge  (\nabla_{g_i}f_i)^* \Psi(TM_i,g_i)$ on {\itshape all of $M$}. Moreover, since both $g_i'$ and $g_i$ are of product form near $\partial M_i$ and $h_i|_{\partial M_i}$ is unimodular, it follows from Proposition \ref{METANOREAL} that $\mathcal R(\mathcal D_i) = \mathcal R(E_i \downarrow M_i,g_i',h_i,\nabla_{g_i'}f_i)$. Therefore, we may assume without loss of generality that $g_i \equiv g_i'$ on all of $M$. \\
Now since the $M_i$ are odd-dimensional with $\partial M_1 = \partial M_2$, we have $\chi(M_1) = \chi(M_2)$. Using this, one proceeds as in \cite[Section 6]{Friedlander:Uni} to show that there exist subdivisions $(\overline{f_i},\overline{g_i})$ of $(f_i,g_i)$ (with $\overline{g_i} = g_i$ near $\partial M_i$), neighborhoods $U_i$ of $\Cr(\overline{f_i}) \cup \partial M_i$ and an isometry
$\theta: (U_1,\overline{g_1}) \to (U_2,\overline{g_2})$ satisfying $\theta(\Cr(\overline{f_1})) = \Cr(\overline{f_2})$, $\theta(M_1) = M_2$ and $\overline{f_2} \circ \theta = \overline{f_1}$. By Lemma \ref{RELINVSUB}, one additionally finds a Hermitian form $\overline{h_i}$ on the bundle $E_i \downarrow M_i$ (with $h_i = \overline{h_i}$ near $\partial M_i$) so that $\overline{ \mathcal D_i} \coloneqq(E_i \downarrow M_i,\overline{g_i},\overline{h_i},\nabla_{\overline{g_i}}\overline{f_i})$ is an admissible system, satisfying \begin{equation} \mathcal R(\mathcal D_i) = \mathcal R(\overline{\mathcal D_i}). \end{equation} Moreover, since the new systems $\overline{\mathcal D_i}$ now also satisfy the assertions of Corollary \ref{rella}.2, we can apply Case $1$ to them and are finished. 
\end{proof}

{\bfseries Proof of Theorem \ref{MAINTHEOREM}:} Let $\mathcal D = (E \downarrow M,g,h,\nabla_{g'} f)$ be an Morse-Smale system of product form, $M$ odd-dimensional, so that $E|_{\partial M} \downarrow \partial M$ is also of determinant class. After pertubing the metric $g$ outside from a neighborhood of $\partial M$, it is because of Proposition \ref{METANO} that we may assume without loss of generality that $g \equiv g'$ outside from a neighborhood of $\partial M$, i.e.\ that $\mathcal D$ is admissible. \\  Choose a Morse-Smale pair $(\hat{f},\hat{g})$ on $\partial M$. Then, \begin{equation*} \mathcal D' \coloneqq (E|_{\partial M} \downarrow \partial M,g|_{\partial M},h|_{\partial M}, \nabla_{\hat{g}} \hat{f}) \end{equation*} is a Morse-Smale system of determinant class. Since $\partial M$ is closed, we have by Theorem \ref{CLOSINV2}\begin{align}\label{ZHANGTRIV}& \mathcal R(\mathcal D') = -\frac{1}{2} \int_{\partial M} \theta(h|_{\partial M}) \wedge (\nabla_{\hat{g}}\hat{f})^* \Psi(T\partial M,g|_{\partial M}) = 0, \end{align} where the last equality follows from the assumption that $h_{\partial M}$ is unimodular, i.e.\ $\theta(h|_{\partial M}) \equiv 0$. \\
Now recall the trivial system $\mathcal D_0 = (E_{\ceals} \downarrow I,g_0,h_0,\nabla_{g_0}f_0)$ over the interval $I = [a,b]$ that we have defined in \ref{TRIVEX} and its relative torsion \begin{align} \mathcal R(\mathcal D_0) = - \frac{\log 2}{2}. \end{align}
Since $\partial M$ is closed and $\partial I = \{a,b\}$, we can form the modified product system \begin{align} \underline{\mathcal D' \times \mathcal D_0} = (E_I \downarrow \partial M \times I, g_I,h_I, \nabla_{\hat{g}_I} \hat{f}_I), \end{align}
with $E_I \coloneqq E_{\partial M} \hat \otimes E_{\ceals}$, $g_I \coloneqq  g_{\partial M} \times g_0$, $\hat{g}_I \coloneqq \hat{g} \times g_0$, $h_I \coloneqq h|_{\partial M}\hat \otimes h_0$ and $\hat{f}_I$ the sum of the Morse functions $\hat{f} + f_0$ that is appropriately modified near the boundary $\partial M \times \{a,b\}$, so that $\underline{\mathcal D' \times \mathcal D_0}$ is a type \rom{2} Morse-Smale system. By Proposition \ref{PRODFORM2}, this system is of determinant class as well and satisfies \begin{align} \mathcal R(\underline{\mathcal D' \times \mathcal D_0}) = \mathcal R(\mathcal D') - \frac{\log{2}}{2} \chi(\partial M,E) \stackrel{\ref{ZHANGTRIV}}{=} - \frac{\log 2}{2} \chi(\partial M) \dim(E).\end{align}
Moreover, as $\theta(h_0) \equiv 0$ and $\theta(h|_{\partial M}) = 0$ by assumption, we retrieve from the product formula \ref{prodvolch} the equality
\begin{align} \theta(h_I) =  \theta(h|_{\partial M} \hat \otimes h_0) =  0.\end{align}
Notice that $\underline{\mathcal D' \times \mathcal D_0}$ is not necessarily an admissible system. This is due to the fact that neither is $g_I$ trivial nor $h_I$ parallel near $\Cr(\hat{f}_I)$. However, since $\Cr(\hat{f}_I)$ is disjoint from $\partial M \times \{a,b\}$, we can pertube the metrics outside of a small neighborhood of $\partial M$ to produce metrics $\widetilde{g_I}$ and $\widetilde{h_I}$, so that $\widetilde{h_I}$ is parallel near $\Cr(f_I)$, and that we have $\widetilde{g_I}  \equiv \hat{g}_I$ outside of a neighborhood of $\partial M$ and near $\Cr(\hat{f}_I)$. By Lemma \ref{UNIMODEX}, the pertubation of the Hermitian form $h_I$ can be performed in such way that still, we have \begin{align}\label{trivform}& \theta(\widetilde{h_I}) \equiv 0, \\& \widetilde{h_I}(p) = h_I(p),  \hspace{.5cm} p \in \Cr(\hat{f}_I).\end{align} For the resulting admissible system $\mathcal D_I \coloneqq (E_I \downarrow \partial M \times I, \widetilde{g_I},\widetilde{h_I}, \nabla_{\hat{g}_I}\hat{f}_I)$, we obtain from Proposition \ref{METANO} that \begin{align}\label{relgood} \mathcal R(\mathcal D_I) = \mathcal R(\underline{\mathcal D' \times \mathcal D_0})  = - \frac{\log 2}{2} \chi(\partial M) \dim(E). \end{align} 
Observe now that by construction, $\mathcal D_I$ and the disjoint union $\mathcal D \sqcup \mathcal D \coloneqq (E \downarrow  M \sqcup E  \downarrow M, g \sqcup g, h \sqcup h, \nabla_{g'}f \sqcup \nabla_{g'}f)$ of $\mathcal D$ with itself are two admissible systems satisfying the hypotheses of Theorem \ref{strongthm}. This allows us to finally conclude as follows:
\begin{align} &2 \mathcal R(\mathcal D ) = \mathcal R(\mathcal D \sqcup \mathcal D )  \nonumber\\&  \stackrel{\ref{power}}{=} \mathcal R(\mathcal D_I) -  \int_{M} \theta(h) \wedge (\nabla_{g}f)^*  \Psi(TM,g)  + \frac{1}{2}\int_{\partial M \times I} \theta(\widetilde{h_I})\wedge (\nabla_{\hat{g}_I}\hat{f}_I)^*  \Psi(T(\partial M \times I),\widetilde{g_I}) \nonumber  \\ & \stackrel{\ref{trivform}}{=}  \mathcal R(\mathcal D_I) -  \int_{M} \theta(h) \wedge (\nabla_{g}f)^*  \Psi(TM,g)  \nonumber \\& \stackrel{\ref{relgood}}{=} -\frac{\log 2}{2} \chi(\partial M) \dim(E) -  \int_{M} \theta(h) \wedge (\nabla_{g}f)^* \Psi(TM,g). \end{align}
This finishes the proof of Theorem \ref{MAINTHEOREM}. $\qed$

\end{document}